\documentclass{amsart}

\usepackage{url,amsmath,amsthm,amssymb,xypic}

\usepackage[colorlinks=true,citecolor=black,linkcolor=black,urlcolor=blue]{hyperref}
\newcommand{\arxiv}[1]{\href{http://arxiv.org/abs/#1}{\texttt{arXiv:#1}}}

\newtheorem{theorem}{Theorem}[section]
\newtheorem*{theorem*}{Theorem}
\newtheorem{proposition}[theorem]{Proposition}
\newtheorem{lemma}[theorem]{Lemma}
\newtheorem{corollary}[theorem]{Corollary}
\newtheorem*{corollary*}{Corollary}
\newtheorem{conjecture}[theorem]{Conjecture}

\theoremstyle{definition}
\newtheorem{example}[theorem]{Example}

\newtheorem{remark}[theorem]{Remark} 

\newtheorem*{conjecture:main}{Conjecture \ref{conj:main}}

\numberwithin{equation}{section}

\def\<{\langle}
\def\>{\rangle}

\def\codim{\operatorname{codim}}

\def\End{\operatorname{End}}
\def\GL{\mathrm{GL}}
\def\Spec{\operatorname{Spec}}
\def\Sym{\operatorname{Sym}}
\def\Hom{\operatorname{Hom}}
\def\Hilb{\operatorname{Hilb}}
\def\Poly{P} 

\def\stab{\rho}

\def\Par{\operatorname{Par}}

\def\AA{\mathbf{A}} 
\def\a{\mathbf{a}} 

\def\B{\mathcal{B}}
\def\b{\mathbf{b}} 
\def\CC{\mathbf{C}}
\def\C{\mathcal{C}} 

\def\kk{\mathbf{k}}
\def\K{\mathcal{K}} 
\def\NN{\mathbb{N}}
\def\OO{\mathcal{O}} 
\def\PP{\mathbb{P}}

\def\RR{\mathbf{R}}
\def\SS{\mathbf{S}}
\def\Ss{\mathfrak{S}}
\def\ZZ{\mathbb{Z}}

\def\rk{\operatorname{rk}} 

\def\cl{\overline}

\def\bw{{{\bigwedge}}} 

\def\FakeDep{\textup{FakeDep}}
\def\Dep{\textup{Dep}}

\let\oldmarginpar\marginpar
\renewcommand\marginpar[1]{\-\oldmarginpar[\footnotesize #1]{\raggedright\footnotesize #1}}

\date{\today} \title{Matrix orbit closures}
\author{Andrew Berget} 
\address[A.~Berget]{Department of Mathematics\\
Western Washington Universty\\
Bellingham, WA}
\email{andrew.berget@wwu.edu}

\author{Alex Fink}
\address[A.~Fink]{School of Mathematical Sciences\\
  Queen Mary University of London\\
  United Kingdom}
\email{a.fink@qmul.ac.uk}

\begin{document}
\maketitle

\begin{abstract}
Let $G$ be the group $\GL_r(\CC) \times (\CC^\times)^n$.
We conjecture that the finely-graded Hilbert series of a
$G$ orbit closure in the space of $r$-by-$n$ matrices
is wholly determined by the associated matroid.
In support of this,
we prove that the coefficients of this Hilbert series
corresponding to certain hook-shaped Schur functions in the $\GL_r(\CC)$ variables
are determined by the matroid, and that the orbit closure
has a set-theoretic system of ideal generators whose combinatorics
are also so determined.
We also discuss relations between these Hilbert series for related matrices,
including their stabilizing behaviour as $r$ increases.
\end{abstract}

\section{Introduction}
In this paper we study a collection of affine varieties that we call
\textit{matrix orbit closures}, obtained as follows.  
Let $\kk$ be an algebraically closed field of characteristic zero
(this assumption can be relaxed in many but not all of our results). 
Pick an $r$-by-$n$ matrix $v$ and consider all
matrices which define a configuration of $n$ points in $\PP^{r-1}$
equivalent to~$v$, that is, which differ from $v$ only by row operations and rescalings
of columns.  The resulting collection is an orbit for the group 
$G=\GL_r(\kk)\times(\kk^\times)^n$,
the former factor acting on the left of the matrix space~$\AA^{r\times n}$
and the latter as diagonal matrices on the right.
The Zariski closure of this set is an (irreducible) affine variety denoted by $X_v$. 

The principal question we seek to
address is to what extent the matroid of~$v$
controls algebraic and geometric properties of~$X_v$. 
It is reasonable to expect some control, on the account that
there is a quotient of~$X_v\subseteq\AA^{r\times n}$ for which the control is very strong.
Let us assume that $v$ has the maximum possible rank~$r$; this is inessential but
makes the statement cleaner.  
Let $G(r,n)$ be the Grassmannian of $r$ dimensional subspaces of~$\kk^n$,
which is the target of the rational map $\pi:\AA^{r\times n}\to G(r,n)$ sending a matrix
to the span of its rows, and inherits the action of the torus $(\kk^\times)^n$.
Then $X_v$ is the closure of the preimage of (the closure of)
the torus orbit containing $\pi(v)$.
These torus orbit closures in~$G(r,n)$ are classified up to isomorphism
by the rank $r$ matroids on a fixed $n$~element set that are realizable over~$\kk$.
The class of $\pi(X_v)$ within the
zeroth torus-equivariant $K$-theory of $G(r,n)$ 
is also a function of the associated matroid, as shown by Speyer~\cite[Proposition~12.5]{speyer}.
Interest in these torus orbit closures antedates the above results:
Klyachko gave a formula for their equivariant cohomology classes in a special case~\cite{klyachko},
and Kapranov undertook a thorough study of the Chow quotient whose points represent them~\cite{kap}.

We prove in a companion paper \cite{cc0} that 
the $G$-equivariant Chow class of~$X_v$ is a function of the matroid of~$v$.
In this paper our specific interest is in a finer invariant,
its class in the $G$-equivariant $K$-theory group,
which contains the same information as the multigraded Hilbert series of the coordinate
ring of~$X_v$.  Our main conjecture is that this refinement adds no
distinguishing power:
\begin{conjecture:main}
  The $G$-equivariant $K$-class of~$X_v$ is determined by the matroid of~$v$.
\end{conjecture:main}
We fully resolve this conjecture here only for rank~2 uniform matroids.

The relation of $X_v$ to its quotient $G(r,n)$ is analogous to that between
\emph{matrix Schubert varieties} and (classical) Schubert varieties:
matrix Schubert varieties are closures of $B\times B$-orbits of square matrices,
$B$ being a Borel group, whereas Schubert varieties are closures of
$B$-orbits in the quotient of this space by~$B$.
The techniques of combinatorial commutative algebra have traction
in the setting of matrix Schubert varieties, which have been used to
great advantage by Fulton~\cite{fulton}, Knutson and Miller \cite{knutsonmiller},
and others since.  The techniques have also been adapted for other
varieties arising in Schubert calculus, such as the Richardson varieties \cite{knutson}.

Further motivation for introducing this $K$-class comes from studying the
general linear group representation generated by the tensor product of
the columns of~$v$, which we call the \textit{tensor module} of~$v$.
The tensor module appears as a multigraded component of the coordinate ring of~$X_v$;
indeed, all other multigraded components are tensor modules of 
configurations obtained from~$v$.
Tensor modules have previously been an object of study in the guise of the
question whether symmetrizations of decomposable tensors are zero, 
which attracted the interest of Gamas \cite{Gamas}, Dias da Silva \cite{dds}, and others.

One of our two main results, Theorem~\ref{thm:hook in tensor module}, 
uses the representation-theoretic perspective on the tensor module
to describe certain coefficients in terms of matroidal combinatorics, namely non-broken circuits.
The coefficients in question are those which are
Schur functions of hook shape, $s_\lambda$ where $\lambda=(n-k+1,1^{k-1})$,
in the $GL_n$ variables and squarefree in the torus variables.
Its translation to the setting of the equivariant $K$-class of~$X_v$, Theorem~\ref{thm:hooks in K},
is particularly pleasant: the corresponding terms are 
a multigraded enumerator of dependent sets of the matroid.
This provides an explicit affirmation of part of Conjecture~\ref{conj:main}:
all hook-shape coefficients in the $K$-class are matroid invariants.

Our other main result, Theorem~\ref{thm:dependent tensors},
gives a generating set for the ideal of $X_v$ up to radical,
the construction of whose generators involves only matroid combinatorics.
When $n$ has a uniform matroid of rank~2 (Proposition~\ref{prop:I'_v is known prime 2}) 
or corank~2 (Proposition~\ref{prop:I'_v is known prime (n-2)})
we prove that our ideal is reduced, i.e.\ is the ideal of~$X_v$.

The structure of our paper is as follows. In Section~\ref{sec:matroids}
we recall some background on matroid theory. 
In Section~\ref{sec:pointsOfOrbitClosures} we classify the points of~$X_v$. 
This enters into the proof of Theorem~\ref{thm:dependent tensors} in the next section.
Section~\ref{sec:K} is dedicated to Conjecture~\ref{conj:main},
on the relationship between the matroid of $v$ and the
$K$-class of $X_v$, and its affirmative resolution in the rank~2 uniform
case, Proposition~\ref{prop:K-polynomial of U_2,n}. 

In Section~\ref{sec:stabilization} we consider the problem of
studying $X_v$ when $v$ is of some rank $r'<r$. Letting $v'$ denote a
matrix whose rows are a basis for the row span of $v$, we relate the
Hilbert series and $K$-polynomials of $X_v$ and $X_{v'}$. We call
$X_v$ the \textit{stabilization} of $X_{v'}$, since $X_v$ is obtained
by embedding $X_{v'}$ in $\AA^{r \times n}$, by adding $r-r'$ rows
equal to zero, and then taking the $\GL_r(\kk)$ orbit of $X_{v'}$. 
A similar operator in cohomology is called a ``raising'' operator in \cite{fnr}.
Stabilization has a trivial effect on Hilbert series (Lemma~\ref{lem:row of zeros});
its translation to $K$-classes (Lemma~\ref{prop:row of zeros}) 
has a less transparent appearance.
Section~\ref{sec:matroid constructions} discusses operations on~$v$
whose effect on the $K$-class we can describe.
One of these is the direct sum of matrices, for which 
stabilization plays a central role;
another is duplicating a column of~$v$.

Finally, we turn to the tensor module in Section~\ref{sec:tensor module},
introducing several fundamental properties of this module including a
Schur--Weyl dual representation.  These are used in the next section to prove
Theorems \ref{thm:hook in tensor module} and~\ref{thm:hooks in K}.

\subsection*{Conventions}
A \emph{variety} is taken to be an integral scheme of finite type over $\kk$.

\section{Matroid theory background}\label{sec:matroids}
White's \textit{Theory of Matroids} \cite{white} serves as an
excellent reference for the matroid theory needed here. For the
convenience of the reader, we gather the required notions in this
section.

A \textbf{matroid} is a simplicial complex $M$ on a finite
\textit{ground set} $E$ whose faces satisfy the following
\textit{exchange axiom}: for faces $I$ and $I'$ of $M$, if $|I|< |I'|$
then there is some $e \in I' \setminus I$ such that $I \cup \{e\}$ is
a face of $M$. Two matroids are isomorphic if they are isomorphic as
simplicial complexes: that is, if there is a bijection between their
ground sets inducing a bijection between their faces. We will refer to
the isomorphism type of a matroid as an \textbf{unlabeled matroid}.

For any matrix $v \in \AA^{r \times n}$ the {matroid} of $v$, denoted
$M(v)$, is the simplicial complex whose faces are those $I \subset
[n]$ such that the columns of $v$ indexed by $I$ are linearly
independent.  Any matrix in the orbit $G\cdot v$ has the same matroid
as $v$. The set of matrices in $\AA^{r \times n}$ with a prescribed
matroid is a subscheme of $\AA^{r \times n}$ called a \textbf{matroid
  stratum} or a \textbf{matroid realization space}.  It is a result
of Sturmfels \cite{sturmfels} that this is not a stratification in any
nice sense (particularly that of Whitney). Worse, a matroid stratum
can contain arbitrarily complicated singularities, a result referred
to as Mn\"ev--Sturmfels universality \cite{vakil}.

Matroids that can be written as $M(v)$ for some $v \in \AA^{r \times
  n}(\kk)$ are said to be \textbf{realizable} over $\kk$. The faces
and non-faces of $M$ are called \textit{independent} and
\textit{dependent} sets, respectively. The minimal dependent sets are
called \textit{circuits} and the maximal independent sets are called
\textit{bases}.

The \textbf{uniform matroid} of rank $r$ on $n$ elements, $U_{r,n}$, is
the matroid with ground set $[n]$ whose bases are all $r$ element
subsets of $[n]$. It is the matroid of a generic element of $\AA^{r
  \times n}$.

We denote the rank of a matrix $v$ by $\rk(v)$. 
The \textbf{rank} $\rk(M)$ of a matroid $M$ 
is the cardinality of a maximal independent set.
In particular, $\rk(M(v)) = \rk(v)$.  
On many occasions we will assume that the rank of 
matroids we deal with is full, i.e., equals $r$.  
In particular, when we state the hypothesis ``$v$ has a uniform
matroid'', we mean uniform of rank~$r$.

For any $v \in \AA^{r \times n}$, its \textbf{Gale dual} is any
$v^\perp \in \AA^{(n-\rk(v)) \times n}$ whose rows form a basis for
the (right) kernel of $v$. Thus, the Gale dual is determined up to the
action of $\GL_{n-\rk(v)}(\kk)$ on $\AA^{(n-\rk(v)) \times n}$. If $v$ has
full rank then Gale duality really is a duality, $\GL_r(\kk)
(v^\perp)^\perp = \GL_r(\kk) v$.  To a matroid $M$ we associate a
\textbf{dual matroid} $M^*$ whose bases are complements of bases of
$M$. If $M(v)$ is the matroid of a matrix $v$, then $M(v)^*$ is the
matroid $M(v^\perp)$ of the Gale dual of~$v$.

The \textbf{direct sum} of two matroids on disjoint sets is the join
of the two simplicial complexes. A matroid is said to be
\textbf{connected} if it is indecomposable with respect to this
operation.  Any matroid $M$ can be written uniquely as a direct sum of
connected matroids, the constituents of which are called the
\textbf{connected components} of $M$. A \textbf{coloop} of $M$ is an
element of $E$ in every base of $M$ and a \textbf{loop} of $M$ is an
element of $E$ in no base of $M$.

The \textbf{rank partition} of $M$ is the sequence of numbers $\lambda(M)
= (\lambda_1,\lambda_2,\dots)$ determined by the condition that for all $k
\geq 1$, its $k$th partial sum is the size of the largest union of $k$
independent sets of $M$. It is a theorem of Dias da Silva \cite{dds}
that $\lambda(M)$ is a partition (i.e., it is weakly decreasing). If $M$
is loop-free then $\lambda(M)$ is the maximum partition $\lambda$ in
dominance order such that $M$ can be $E$ can be partitioned into
independent sets of sizes $\lambda_1,\lambda_2,\dots$\ .

The \textbf{restriction} of $M$ to a subset $J \subset E$, denoted
$M|J$, consists of those independent sets belonging to $J$.
The \textbf{contraction} of $M$ by $J$ is $(M^*|J^{\rm c})^*$,
where $J^{\rm c} = E \setminus J$, and is denoted $M/J$. If $M=M(v)$ is
realizable then $M/J$ is obtained as follows. Let $A \in \End(\kk^r)$
be a matrix whose kernel is spanned by $\{v_j : j \in J\}$ and is
generic with respect to this property. Then $M/J$ is the matroid of
$Av$, with columns $J$ deleted.

If there is a matroid $M'$ with ground set $E' \supset E$ such that $M
= M'|E$ then $M'/(E' \setminus E)$ is said to be a \textbf{quotient}
of $M$. It follows that every quotient of a realizable matroid is
again realizable.

Let the indicator vector of a subset $B$ of $[n]$ be $e_B = \sum_{i
  \in B} e_i$.  The \textbf{matroid (base) polytope} $\Poly(M)$ of a
matroid $M$ with ground set $[n]$, essentially due to Edmonds~\cite{edmonds}, 
is the convex hull of the indicator
vectors of the bases of $M$ in $\RR^n$.  It is a theorem of 
Gel'fand, Goresky, MacPherson and Serganova~\cite{GGMS}
that, among non-empty polytopes $P$ with vertices
chosen from the set $\{e_B : B \subset [n]\}$, matroid polytopes are
exactly those that lie in a plane where the coordinates sum to a
positive integer and every edge of $P$ has the form
$\operatorname{conv}\{e_B,e_{B \cup j \setminus i}\}$ for some
$B\subset[n]$ and $i \in B$, $j \notin B$.

\section{The points of a matrix orbit closure}\label{sec:pointsOfOrbitClosures}
In this section we discuss the geometry of the matrix orbit closures
$X_v$ with respect to the $G$ orbits they comprise.

\begin{proposition}
  The closure of a $G$-orbit in $\AA^{r \times n}$ is an
  irreducible affine variety. If $v$ has a matroid of rank $r$ with
  $c$ connected components, then 
  $$
  \dim( X_v ) = r^2 +n -c.
  $$
\end{proposition}
\begin{proof}
  Since $G$ is a connected linear algebraic group the first claim
  follows. The second follows since the stabilizer of $v$ is seen to
  be a $c$-dimensional torus inside the diagonal torus of $G$.
\end{proof}

Let $(\AA^{r \times n})^{\rm fr}$ denote the open subvariety of full
rank matrices in $\AA^{r \times n}$. There is a $\GL_r$ bundle $\pi:
(\AA^{r \times n})^{\rm fr} \to G(r,n)$, which takes a matrix to its
row span. Consider the case that $v \in (\AA^{r \times n})^{\rm
  fr}$. Then $\cl{\pi(v) T} \subset G(r,n)$ is the (normal) toric
variety associated to the matroid polytope of $M(v)$. The $T$-orbits
in $\cl{\pi(v)T}$ are in bijection with the faces of the matroid base
polytope $P(M(v))$. One can give a combinatorial description of the
faces of the matroid polytope as follows \cite[Proposition~2]{ak}. Let
$S_\bullet$ be a flag of subsets
\[
\emptyset = S_0 \subset S_1 \subset \dots \subset S_k \subset S_{k+1}
= [n].
\]
Every face of $\Poly(M(v))$ is of the form $\Poly(M(v)_{S_\bullet})$
where
\[
M(v)_{S_\bullet} = \bigoplus_{i=1}^{k+1} (M(v)|S_i)/S_{i-1}.
\]
Two different flags can produce the same matroid, but there is only
one $T$-orbit in $\cl{\pi(v)T}$ with a given matroid.  A realization
of this result in terms of torus orbit closures is obtained as
follows. Rescale column $i \in S_j \setminus S_{j-1}$ of $v$ by
$s^{j-1}$. Projecting this matrix into $G(r,n)$ we obtain a subspace
$\pi(v)\lambda(s)$, where $\lambda(s)$ is a one-parameter subgroup of
$T$, i.e., an element of $T(\kk(\!(s)\!))$. Here $\kk(\!(s)\!)$ is the
field of Laurent series in $s$ over $\kk$. Taking the limit $\lim_{s
  \to 0} \pi(v)\lambda(s)$ yields a point of $\cl{\pi(v)T}$ with
matroid $M(v)_{S_\bullet}$. Every $T$-orbit in $\cl{\pi(v) T}$ is
reached in this way and so our argument is complete.

The pullback $\pi^{-1}( \lim_{s \to 0} \pi(v)\lambda(s) )$ is the
$G$-orbit of a full rank matrix in $X_v$ whose matroid is
$M(v)_{S_\bullet}$. We call any such matrix a \textbf{projection of
  $v$ along the flag $S_\bullet$}. As before, there is only one
$G$-orbit in $X_v$ whose points have a prescibed matroid of the form
$\bigoplus_{i=1}^k M(v)|S_i/S_{i-1}$.

The next result shows that all elements of $X_v$ are
obtained by projecting $v$ along some flag and applying some element
$g \in \End(\kk^r)$ on the left.
\begin{proposition}\label{prop:projecting full rank configurations}
  Suppose that $v$ has rank $r$ and $w \in X_v$ is a matrix of rank
  less than $r$. Then there is a matrix $w' \in X_v$ whose rank is
  that of $v$, and $w = g w'$ for some singular $g \in \End( \kk^r)$.
\end{proposition}
\begin{proof}
  Let $V = \AA^{r \times n}$ and suppose that $w$ has rank
  $\ell$. After applying an element of $\GL_r$, and relabeling the
  columns of our matrices, we may assume that $w$ has a row equal to
  zero and its first $\ell$ columns are the first $\ell$ standard
  basis vectors. By the valuative criterion for properness, there is
  an element $(g(s),t(s))$ of $G(\kk(\!(s)\!))=\GL_r(\kk(\!(s)\!)) \times T(\kk(\!(s)\!))$ such
  that $g(s)vt(s) \in \kk[[s]] \otimes_\kk V$ and
  \[ 
  g(s) v t(s) \equiv w \mod s. 
  \]

  Applying an element of $\GL_r( \kk[[a]])$ we may assume that the
  first $\ell$ columns of $g(s)vt(s)$ are the first $\ell$ standard
  basis vectors. Let $\nu_i$ be the least of the non-negative integers
  that appears as an exponent in row $i$ of $g(s) v t(s)$. The limit
  of
  \[ 
  \operatorname{diag}(s^{-\nu_1},\dots,s^{-\nu_r}) (g(s) v t(s)) 
  \]
  as $s \to 0$ gives an element $w'$ that has rank strictly larger
  than $w$. If the rank of $w'$ is not the rank of $v$ then, by
  induction, there is some $w'' \in X_v$ of rank $r$ and $g
  \in \End(\kk^r)$ such that $g w'' = w'$. Applying an element of
  $\End(\kk^r)$ that zeros out the appropriate rows, we bring $w'$ to
  $w$.
\end{proof}

\begin{corollary}\label{cor:matroids appearing in an orbit closure}
  If $w \in X_v$ then there is a flag of sets $S_\bullet$
  such that the matroid of $w$ is a quotient of
  \[
  \bigoplus_{i=1}^{k+1} (M(v)|S_i)/S_{i-1}.
  \]
  Conversely, every quotient of such a matroid occurs as the matroid
  of some $w \in X_v$.
\end{corollary}
\begin{proof} Combining the remarks above about faces of the matroid
  polytope $P(M(v))$ with Proposition~\ref{prop:projecting full rank
    configurations}, we obtain the first claim. The converse follows
  since the contraction of a matroid by a set of elements is realized
  by applying an element of $\End(\kk^r)$ to a vector configuration in
  $\kk^r$. This means that every quotient of $\bigoplus_{i=1}^{k+1}
  (M(v)|S_i)/S_{i-1}$ is the matroid of a point in the
  $\End(\kk^r)$-orbit of $\pi^{-1}(\cl{\pi(v) T})$.
\end{proof}

\begin{example}
  The correspondance between matroids and orbits in $X_v$
  is not in general bijective as the following example shows. If
  \[
  v =
  \begin{bmatrix}
    1 & 0 & 0 & 1 & 1 \\
    0 & 1 & 0 & 1 & 1 \\
    0 & 0 & 1 & 0 & 1
  \end{bmatrix}
  \]
  Then for every $\mu \in \kk$,
  \[
    \begin{bmatrix}
    1 & 0 & 0 & 1 & 1 \\
    0 & 1 & \mu-1 & 1 & \mu \\
    0 & 0 & 1 & 0 & 1
  \end{bmatrix} \in X_v^\circ  \implies
  \begin{bmatrix}
    1 & 0 & 0 & 1 & 1 \\
    0 & 1 & 0 & 1 & \mu \\
    0 & 0 & 0 & 0 & 0
  \end{bmatrix} \in X_v.
  \]
  The matrices of the latter form are all projectively
  inequivalent. This stands in contrast to the situation with
  $\cl{\pi(v) T}$, where orbits are in bijection with the
  matroids of the points in the orbit closure.
\end{example}

\section{The ideal of a matrix orbit closure}\label{sec:ideal}
Let $R = \kk[x_{ij} : i \in [r], j \in [n]]$, and regard $\Spec(R) =
\AA^{r \times n}$. Since a matrix orbit closure $X_v$ is irreducible
in $\AA^{r \times n}$, it is the vanishing locus of a prime ideal $I_v
\subset R$. In this section we discuss this ideal. Our main result is
Theorem~\ref{thm:ideal generators}, which gives a finite generating
set for an ideal $I'_v$ given by minors of certain matrices, for which
$\sqrt{I'_v} = I_v$.

\subsection{The ideal $I'_v$}
We now give the polynomial conditions for a matrix to lie in
$X_v$. Recall from Section~\ref{sec:matroids} the notion of
Gale duality. For $v \in \AA^{r \times n}$, its Gale dual is any
$v^\perp \in \AA^{(n-\rk(v)) \times n}$ whose rows form a basis for
the kernel of $v$. For any $w=(w_1,\dots,w_n) \in X_v^\circ$, the
vectors
\[
  w_1 \otimes v^\perp_1, \quad w_2 \otimes v^\perp_2 ,\quad \dots ,
  \quad w_n \otimes v^\perp_n
\]
are linearly dependent. This can be seen by expanding a linear
combination in the standard basis of $\kk^r \otimes \kk^{n-r}$. By
continuity this holds for any $u \in X_v$. More is true:

\begin{proposition}[{Kapranov \cite{kap}}]\label{prop:kap-tensors}
  Suppose that $w \in \AA^{r \times n}$ has a connected matroid of
  full rank. If the collection of tensors
  \[
  w_1 \otimes v^\perp_1, \quad w_2 \otimes v^\perp_2 ,\quad \dots ,
  \quad w_n \otimes v^\perp_n
  \]
  forms a circuit in $\kk^r \otimes \kk^{n-r}$ then $w \in X_v^\circ$.
\end{proposition}

For a subset $J$ of $[n]$, let $v_J$ be the submatrix of $v$ on the
columns indexed by $J$, so that the rank $\rk(M|J)$ in the matroid of
$v$ is the dimension of the span of these columns in $\kk^r$. The Gale
dual of $v_J$ is not $(v^\perp)_J$, but it is a projection of this
configuration.  This fact is matroidally manifested by the equality $(M|J)^*
= M^*/J^{\rm c}$ where $J^{\rm c}$ is the complement of $J$ in the ground set of
$M$.

\begin{theorem}\label{thm:dependent tensors}
  For any $v \in \AA^{r \times n}$, a matrix $w$ is in $X_v$ if and only if 
  for every $J =\{j_1,\dots,j_\ell\}\subset [n]$,
  the tensors
  \begin{align}\label{eq:tensors}
  \{ w_{j_i} \otimes (v_J^\perp)_i : i =1 ,\dots, \ell\},    
  \end{align}
  are linearly dependent.
\end{theorem}
The proof of the theorem can be found in Section~\ref{proof:dependent
  tensors} below.

An immediate consequence of Theorem~\ref{thm:dependent tensors} 
is the next theorem, giving set-theoretic equations for $X_v$. 
Let $x$ denote the matrix of variables
$x_{i,j}$, and let $x_j$ denote the $j$-th column
$(x_{1,j},\dots,x_{r,j})^t$ of $x$. For each subset $J =
\{j_1,\dots,j_\ell\} \subset [n]$ we form the matrix $x_J \odot
v_J^\perp$, whose columns are the tensors $x_{j_i} \otimes
(v_J^\perp)_i \in R^r \otimes \kk^{n-\rk(v_J)}$.
There exists a linear dependence among the columns
of $x_J \odot v_J^\perp$ if and only if all its size $|J|$ minors vanish.
As such:

\begin{theorem}\label{thm:ideal generators}
  Let the size $|J|$ minors of the matrices $x_J \odot v_J^\perp$, $J
  \subset [n]$, generate the ideal $I'_v \subset R$.  Then $\sqrt{I'_v} = I_v$.
\end{theorem}

\begin{remark}\label{rmk:ideal generators}
  There are two special cases that occur when applying this
  result. The first occurs when a subconfiguration $v_J$ consists of
  linearly independent vectors. In this case $v_J^\perp$ is a
  configuration of $n$ null vectors. We interpret $x_J \otimes
  v_J^\perp$ to be the zero matrix in this case. The second special
  case is when $v_J$ has $U_{|J|-1,|J|}$ as its matroid. In this case
  the dimensions of the matrix $x_J \odot v_J^\perp$ are
  $(|J|-1)$-by-$|J|$, and hence all its size $|J|$ minors vanish.
\end{remark}

\subsection{On the primality of $I'_v$}

We first make a conjecture.
\begin{conjecture}
  The ideal $I'_v$ is equal to $I_v$.
\end{conjecture}
We can prove this conjecture in two special cases.

\begin{proposition}\label{prop:I'_v is known prime (n-2)}
  Suppose $v \in \AA^{(n-2) \times n}$ and that $v$ has a uniform
  matroid of rank $n-2$. Then $I'_v = I_v$ and $R/I'_v$ is a
  Cohen--Macaulay ring.
\end{proposition}
\begin{proof}
  The hypothesis on $v$ ensures that it has a full dimensional
  orbit. The codimension of the orbit closure is thus $n-3$.  It
  follows that the codimension of $I'_v$ is $n-3$.

  If $v \in \AA^{(n-2) \times n}$ then Remark~\ref{rmk:ideal
    generators} implies that $I'_v$ is generated by the size $(n-2)$
  minors of $x \odot v^\perp$ --- it is a determinantal ideal. Since
  $x \odot v^\perp$ has dimension $2(n-2)$-by-$n$, $I'_v$ has the
  expected codimension. We apply \cite[Corollary~4]{eagonHochster} to
  conclude that $I'_v$ is prime and hence $I'_v = I_v$. The cited
  result also implies that $R/I'_v$ is a Cohen--Macaulay ring.
\end{proof}

The second case of primality is Gale dual to the first.

\begin{proposition}\label{prop:I'_v is known prime 2}
  Suppose that $v \in \AA^{2 \times n}$ and that $v$ has a uniform
  matroid of rank $2$. Then $I'_v = I_v$ and $R/I'_v$ is a
  Cohen--Macaulay ring.
\end{proposition}

The proof of this result follows by constructing a third ideal from
$v$ with the desired properties. This ideal is contained in $I'_v$ and
we will show that the former ideal cuts out $X_v$. Specifically, let $I''_v$ denote the ideal generated by the
size $4$ minors of the $4$-by-$n$ matrix
\[
x \odot v = \begin{bmatrix}
    \begin{bmatrix}
      x_{11}\\x_{21}
    \end{bmatrix} \otimes v_1 &
    \begin{bmatrix}
      x_{12}\\x_{22}
    \end{bmatrix} \otimes v_2
    &\cdots& 
    \begin{bmatrix}
      x_{1n}\\x_{2n}
    \end{bmatrix} \otimes v_n
  \end{bmatrix}.
\]
Given two integers $a <b \in [n]$, we let $p_{ab}(v)$ denote the
determinant of the $2$-by-$2$ submatrix of $v$ with columns $a$ and
$b$. Similarly define $p_{ab}(x)$. It is an immediate calculation that
the minors of $x \odot v$ are all of the form
\[
p_{ab}(v) p_{cd}(v) p_{ac}(x) p_{bd}(x) - p_{ac}(v) p_{bd}(v)
p_{ab}(x)p_{cd}(x).
\]
This polynomial is obtained from the equality of the cross ratio
\[
\frac{p_{ab}(v) p_{cd}(v)}{p_{ac}(v) p_{bd}(v)}
=
\frac{p_{ab}(gvt) p_{cd}(gvt)}{p_{ac}(gvt) p_{bd}(gvt)},
\]
which holds on the orbit $X_v^\circ$, which is open in its closure.
\begin{proposition}
  The vanishing locus of $I''_v$ is $X_v$.
\end{proposition}
\begin{proof}
  The result follows by induction on $n$. If $n=4$ then $\codim
  X_v = 1$, and so $I_v$ must be principal. Since $I''_v$
  is principal and $I_v$ cannot be generated by a linear form or a
  constant, we must have equality.
  
  Suppose that $n>4$ and let $w = (w_1,\dots,w_n)$ be a $2$-by-$n$
  matrix in the vanishing locus of $I''_v$. If $w$ has a zero column
  then $w \in X_v$ by induction on $n$. Hence, we assume
  that no column of $w$ is zero. Assume that $w$ has a pair of
  parallel columns. The fact that $w$ vanishes at all the generators
  of $I''_v$ implies that at least $n-1$ of the columns of $w$ are
  parallel. A simple calculation proves that $w \in X_v$.

  Finally, assume that $w$ has no parallel columns. By induction we
  can bring the first $n-1$ columns of $w$ to those of $v$ by a
  projective transformation. Since the first, second, third and last
  column of $w$ have a presribed cross ratio, we see that $w_n$ must
  be a non-zero scalar multiple of $v_n$.
\end{proof}

\begin{proof}[Proof of Proposition~\ref{prop:I'_v is known prime 2}]
  By \cite[Corollary~4]{eagonHochster} we know that $I''_v$ is
  prime. A short calculation gives $I''_v \subset I'_v$ and we
  conclude that $I''_v = I'_v = I_v$. The cited result also yields the
  fact that $R/I''_v$ is a Cohen--Macaulay ring.
\end{proof}


\subsection{Proof of Theorem~{\ref{thm:dependent tensors}}}\label{proof:dependent tensors}
The ``only if'' direction of the theorem is true by the discussion
proceeding Proposition~\ref{prop:kap-tensors}. Our proof of the ``if'' direction
is by induction.
\begin{lemma}
  Suppose that for every $r' < r$ Theorem~\ref{thm:dependent tensors}
  is true for $\AA^{r' \times n}$. Then to prove the theorem for
  $\AA^{r \times n}$, we may assume that $v \in (\AA^{r \times
    n})^{\rm fr}$.
\end{lemma}
\begin{proof} 
  Suppose that $v \in \AA^{r \times n}$ has rank $r'< r$ and that $u
  \in \AA^{r \times n}$ has rank larger than $r'$.  Replacing $v$ with a
  matrix in $\GL_r v$ we may assume that the last $r-r'$ rows of $v$
  are zero. Let $J=\{j_1,\dots,j_{r'+1}\} \subset [n]$ denote a set of
  indices of size $r'+1$ such that $u_J$ has rank $r'+1$. Not every
  column of $v_J$ can be a coloop of the matroid $M(v)|J$, so by
  throwing away elements of $J$, assume that $v_J$ is coloop
  free. Consider the tensors
  \[
  u_{j_1} \otimes (v_J^\perp)_1, \quad u_{j_2} \otimes (v_J^\perp)_2,
  \dots, \quad u_{j_{r'+1}} \otimes (v_J^\perp)_{r'+1},
  \]
  These are linearly independent, since the columns of $u_J$ are
  linearly independent, except in the case that one of the columns of
  $v_J^\perp$ is zero. However, every column of $v_J^\perp$ is
  non-zero since $M(v)|J$ is coloop free.

  Suppose that for every $J \subset [n]$, the tensors in
  \eqref{eq:tensors}, with $u$ taking the place of $w$,
  are linearly dependent. Then $u$ has rank at most
  $r'$. Replacing $u$ with a $\GL_r$-translate, we may assume that the
  latter $r-r'$ rows of $u$ \textit{and} $v$ are zero. Ignoring the
  latter $r-r'$ rows of $u$ and $v$, we can appeal to the truth of
  Theorem~\ref{thm:dependent tensors} for $\AA^{r' \times n}$, thus
  proving the lemma.
\end{proof}

\begin{lemma}
  Suppose that for every $n' < n$ Theorem~\ref{thm:dependent tensors}
  is true for matrices in $\AA^{r \times n'}$. Then, to prove the
  theorem for $v \in \AA^{r \times n}$, we may assume that $v$ has a
  connected matroid.
\end{lemma}
\begin{proof}
  Suppose that $v$ has a disconnected matroid and, after permuting
  columns of $v$, write $v = [v'\ v'']$ where $M(v) = M(v') \oplus
  M(v'')$. After applying an element of $\GL_r$, we may assume that
  the first $r'$ rows of $v''$ are zero, and the latter $r'' =
  r-r'$ rows of $v'$ are zero, so that $v$ is a direct sum of
  matrices. From this we see that $v^\perp$ is also a direct sum of matrices.

  Pick $w \in \AA^{r \times n}$ satisfying conditions
  \eqref{eq:tensors} and write $w = [w'\ w'']$, where $w'$ and $w''$
  have the same numbers of columns as $v'$ and $v''$. It follows that
  $w' \in X_{v'}$ and $w'' \in X_{v''}$ by our induction hypothesis.

  By Proposition~\ref{prop:projecting full rank configurations} there
  are configurations $\tilde w' \in X_{v'}$, $\tilde w'' \in X_{v''}$,
  of rank $r'$ and $r''$ respectively, and matrices $g,h
  \in \End(\kk^r)$ such that $g \tilde w' = w'$ and $h \tilde w'' =
  w''$. We may assume that $\tilde w''$ has its latter $r'$ rows equal
  to zero, and that the first $r''$ rows of $\tilde w'$ are
  zero. Thus, $[\tilde w'\ \tilde w''] \in X_v$. Taking the first $r'$
  columns of $g$ and the latter $r''$ columns of $h$ and forming a new
  matrix $A$ from these, we have that $A[\tilde w'\ \tilde w''] = w$,
  which is thus in $X_v$.
 \end{proof}
We are now in a good place to prove the theorem.
\begin{proof}[Proof of Theorem~\ref{thm:dependent tensors}]
  We have reduced to the case that $v$ has full rank and a connected
  matroid. This implies that the orbit $X_v^\circ$ has dimension $r^2
  + n-1$.  If $n=r+1$ then $r^2+n-1 = r(r+1)$ and so $X_v = \AA^{r
    \times (r+1)}$. The tensors in \eqref{eq:tensors} are dependent by
  a dimension count, so the theorem is true in this case. We will
  assume that $n > r+1$ below.

  Assume that we have $u \in \AA^{r \times n}$ such that for all $J
  \subset [n]$, the tensors in \eqref{eq:tensors} are linearly
  dependent. We prove that $u \in X_v$ by induction on $n$. 

  We start with the case when the rank of $u$ is less than $r$.
  Assume that the last column of $u$ is non-zero, since we are done if
  it is. Applying elements of $\GL_r$ we may assume that the $r$th row
  of $u$ is all zeros and the last entry of $v_n$ is non-zero.

  Let $u' = (u_1,\dots,u_{n-1})$ and likewise for $v$. By induction on
  $n$, we know that there is some element $(g(s),t(s)) \in
  \GL_r(\kk(\!(s)\!)) \times (\kk(\!(s)\!)^\times)^{n-1}$ such that
  $g(s)\,v'\,t(s)$ has coordinates in $\kk[[s]]$ and
  \[
  g(s)\,v'\,t(s) \equiv u' \mod s.
  \]
  Since the bottom row of $u'$ is all zeros we can replace the bottom
  row of $g(s)$ with $(0,\dots,0,s^m)$, $m >\!\!\!>0$, and obtain the
  same reduction modulo $s$. Now let $t'(s) = (t(s),s^{-m}) \in
  (\kk(\!(s)\!)^\times)^n$ and consider the matrix
  \[
  g(s)\, v\, t'(s).
  \]
  Setting $s=0$ yields a matrix whose first $n-1$ columns agree with
  those of $w$ and whose last column is the $r$th standard basis
  vector of $\kk^r$. Applying the element of $\End(\kk^r)$ that fixes
  the first $r-1$ basis vectors and sends the last to $u_n$, we bring
  this matrix to $u$. We conclude that $u \in X_v$.

  Suppose that $u$ has rank $r$. If $u$ has a connected matroid then
  Proposition~\ref{prop:kap-tensors} shows that $u \in X_v^\circ$. We
  thus reduce to the case that $u$ has a disconnected matroid. Our
  goal is to show that $M(u)$ has a connected component $K$ such that
  the orbit of $u_K$ equals the orbit of $v_K$.

  For any $J \subset [n]$, the rank of $(v_J)^\perp$ is $\dim
  \ker(v_J)=|J| - \rk_{M(v)}(J)$. If $K \subset J$, then the
  restriction of $(v_J)^\perp$ to the columns indexed by $K$ has rank
  $\dim \ker(v_J) - \dim \ker(v_{J \setminus K})$.

  Since the tensors in \eqref{eq:tensors} are dependent, it follows
  that for any $J \subset [n]$ there is a connected component $K$ of
  $M(u)$ with $J \cap K$ non-empty and $\dim \ker(v_J) - \dim \ker (v_{J
    \setminus K})$ linearly independent dependences among $u_{J \cap
    K}$. That is, $$\dim \ker(u_{J \cap K}) \geq \dim\ker(v_J) - \dim
    \ker(v_{J \setminus K}),$$ and hence
  \begin{equation}\label{eq:constructing components}
    \rk u_{J \cap K} \leq \rk v_J - \rk v_{J \setminus K}.
  \end{equation}
  for some connected component $K$ of $M(u)$. 
H
  Applying \eqref{eq:constructing components} with $J = [n]$ we obtain a
  component $K_1$ of $M(u)$. Apply \eqref{eq:constructing components}
  again with $J = [n] \setminus K_1$ and obtain a connected component
  $K_2$ of $M(u)$. Continue in this way to obtain $K_1, \dots,
  K_\ell$, an ordering of the components of $M(u)$. Summing the
  inequalities obtained from \eqref{eq:constructing components}
  yields,
  \begin{multline*}
    \rk(u_{K_1}) +\rk(u_{K_2})+ \dots +\rk(u_{ K_\ell}) \leq (\rk v -
    \rk v_{[n] \setminus K_1})\\+
(\rk v_{[n] \setminus K_1}- \rk v_{[n] \setminus K_1 \cup K_2}) + \dots + (\rk v_{K_\ell} - \rk
    v_{\emptyset})
  \end{multline*}
  The left and right sides of this are both $r$ and hence all the
  inequalities above are all equalities. It follows that $\rk
  (u_{K_\ell}) =\rk(v_{K_\ell})$. We know that $u_{K_\ell} \in
  X_{v_{K_\ell}}$ by the induction hypothesis, and thus $v_{K_\ell}$
  is connected because $u_{K_\ell}$ is. We conlcude from
  Proposition~\ref{prop:kap-tensors} that the orbit of $v_{K_\ell}$
  equals the orbit of $u_{K_\ell}$. We thus take $u_{K_\ell} =
  v_{K_\ell}$.

  Setting $K_\ell^{\rm c} = [n] \setminus K_\ell$, there is some $g(s),t(s)$
  such that $g(s)\, v_{K_\ell^{\rm c}}\, t(s) \equiv u_{J^{\rm c}}$ mod
  $s$. Since the first $\rk(u_{K_\ell})$ rows of $u_{{K_\ell}^{\rm c}}$ can be taken
  to be zero, we replace the first $\rk(u_{K_\ell})$ rows of $g(s)$ with the
  corresponding rows of $s^m\mathrm{Id}_r$ for $m\gg0$, and apply
  $g(s),(s^{-m},\dots,s^{-m},t(s))$ to $v$. The result is $u$.
\end{proof}

\section{Multigraded Hilbert series and $K$-polynomials}\label{sec:K}
In this section we define the multigraded Hilbert series and 
$K$-polynomial of a $G$-equivariant $R$-module. We then propose the
fundamental question of our work, which is on the matroid invariance
of the $K$-polynomial of the coordinate ring of a matrix orbit
closure. We then give a formula for this $K$-polynomial when $v \in
\AA^{2 \times n}$ has a uniform rank $2$ matroid.

\subsection{Background on Hilbert series}
Let $R$ denote the polynomial ring $$\kk[x_{ij}: i \in [r], j \in
[n]]$$ and regard $\AA^{r\times n}$ as $\Spec R$. $R$ is graded by
$\ZZ^r \times \ZZ^n$, the degree of $x_{ij}$ being $a_i + b_j $,
where $a_1, \dots, a_r, b_1, \dots, b_n$ are the standard basis
vectors of $\ZZ^r \times \ZZ^n$. The grading group should be thought
of as the weight lattice of the maximal torus in $G=\GL_r(\kk) \times T$
obtained as (the diagonal torus of $\GL_r(\kk))\times T$.

Any finitely generated graded $R$-module $M = \bigoplus_{(\a,\b) \in
  \ZZ^r \times \ZZ^n} M_{(\a,\b)}$ has \textbf{Hilbert series}
\[
\Hilb(M) =   \sum_{(\a,\b) \in \ZZ^r \times \ZZ^n}
\dim_\kk(M_{(\a,\b)} ) u^\a t^\b \in \ZZ[[u_1^{\pm 1},\dots,u_r^{\pm
  1},t_1^{\pm 1},\dots,t_n^{\pm 1}]].
\]
By \cite[Theorem~8.20]{miller-sturmfels}, there is a Laurent
polynomial $\K(M;u,t)$ such that
\[
\Hilb(M)=\frac{\K(M;u,t)}{\prod_{i=1}^r \prod_{j=1}^n (1-u_it_j)},
\]
and we refer to this polynomial as the \textbf{$K$-polynomial} of
$M$. In particular, if $X \subset \AA^{r \times n}$ is a closed
subvariety with defining ideal $I$, we write $\K(X;u,t)$ for the
$K$-polynomial of $R/I$. We will often write the $K$-polynomials as
$\K(M)$, $\K(X)$, etc..

The ring $R$ has the action of $G$ given by $((g,t) \cdot f)(v) = f(
g^{-1} v t)$. The decomposition of $R$ into its various graded pieces
$R_{(\a,\b)}$ is a refinement of the irreducible decomposition of $R$
as a $G$-module; it is precisely the refinement into weight spaces. It
is important to take care that the gradation and weight space
decompositions have the property that if $f \in
R_{(\a,\b)}$, then $f$ has $\GL_r(\kk)$-weight $-\a$.
The arguably more natural convention of setting $\deg(x_{ij}) = b_j -
a_i$ results in ugly formulas and does not agree with the standard
grading of $R$. Thus, given a $G$-equivariant graded module $M$ we
pass back and forth between its character, as a $G$-module, and its
Hilbert series by inverting all the $u$ variables.

Let $K^0_G(\AA^{r \times n})$ denote the Grothendieck group of
$G$-equivariant coherent $R$-modules. Since $\AA^{r \times n}$ is a
trivial vector bundle over a point, there is a natural
identification 
of $K^0_G(\AA^{r \times n})$ with the Grothendieck group of rational
representations of $G$. Indeed, we have
\begin{align*}
K^{0}_{G}(\AA^{r\times n})
  &= \ZZ[u_1,\ldots,u_r,t_1,\ldots,t_n][u_1^{-1},\ldots,u_r^{-1},t_1^{-1},\ldots,t_n^{-1}]^{\Ss_r }
\\&= \ZZ[e_1(u),\ldots,e_r(u),t_1,\ldots,t_n][e_r(u)^{-1},t_1^{-1},\ldots,t_n^{-1}],
\end{align*}
where the symmetric group $\Ss_r$ acts on the $u$ variables. Here we
have written $e_i(-)$ for the $i$\/th elementary symmetric polynomial in its
arguments. Under this identification, the $K$-polynomials $\K(X)$ and
$\K(\mathcal{E})$ represent the equivariant $K$-classes of the
structure sheaf of $X$ and the global sections of $\mathcal{E}$,
respectively.

In what follows, we will employ the following standard notation.
\begin{itemize}
\item $\Par_r$ is the set of partitions $\lambda=(\lambda_1
  \geq \dots \geq \lambda_r \geq 0)$ of length at most~$r$.
\item $s_\lambda(x)$ is a Schur polynomial in the list of variables
$x=(x_1,\dots,x_k)$.
\item $e_k(x) = s_{1^k}(x)$ is an elementary symmetric
polynomial.
\item $h_k(x) = s_k(x)$ is a complete homogeneous symmetric
polynomial.
\end{itemize}
It will also be useful to give meaning to $s_\lambda(x_1,\ldots,x_k)$
when $\lambda$ is a $k$-tuple of integers, possibly negative, that need not
be a partition.  We do this using the determinantal formula
\[
s_\lambda(x_1,\ldots,x_k) = 
\frac{\det((x_i^{\lambda_j+k-j})_{i,j})}{\det((x_i^{k-j})_{i,j})}.
\]
In particular, the set 
\[\{s_\lambda(x_1,\ldots,x_k) : \mbox{$\lambda\in\ZZ^k$ is
  non-increasing}\}\] is a basis for the symmetric Laurent polynomials
in $x_1,\ldots,x_k$.

\subsection{Two fundamental problems} The main motivation behind our
work is 
\begin{enumerate}
\item to determine if $\K(X_v)$ is determined by $M(v)$ alone, and
\item when it is, determine exact formulae for $\K(X_v)$ in terms of
  $M(v)$.
\end{enumerate}
The first propblem is opposed to whether $\K(X_v)$ is determined by
``higher order'' geometric properties of the configuration $v$. In all
computed cases, $\K(X_v)$ only varies with $M(v)$. We conjecture that
the answer to question (1) is positive:
\begin{conjecture}\label{conj:main}
  $\K(X_v)$ is determined by $M(v)$.
\end{conjecture}
One result motivating the conjecture is that if we replace $X_v$ with
its Grassmannian analog $\cl{\pi(v)T}$ then the result is
true. Specifically, a result of Speyer \cite[Proposition~12.5]{speyer}
says that the $T$-equivariant $K$-theory class of the structure sheaf
of $\cl{\pi(v)T}$ is determined by the matroid $M(v)$. This result
follows by equivariant localization, a tool which is difficult to bring to bear on
the equivariant $K$-theory of $\AA^{r \times n}$.

The second motivating problem appears to be very difficult to answer
for arbitrary $v$. In the rest of this section we give an answer when
$v \in \AA^{2 \times n}$ has matroid $U_{2,n}$. This follows since we
know that the ideal of $X_v$ is determinantal, and so we can compute a
free resolution of its coordinate ring using known results. In
Section~\ref{ssec:hooks} we determine the coefficient of
$s_\lambda(u)t^\b$ when $\b$ is a $\{0,1\}$-vector and $\lambda$ is a
single column. The answer is always $\pm 1$, and the proof of this
relies on results about broken circuit complexes and multivariate
Tutte polynomials.

\subsection{The $K$-polynomial of $X_v$ when $M(v)=U_{2,n}$} 
  Proposition~\ref{prop:I'_v is known prime 2} allows us to explicitly
  determine the $K$-polynomial $\K(X_v)$ when $M(v) = U_{2,n}$.
\begin{proposition}\label{prop:K-polynomial of U_2,n}
  Let $v\in\AA^{2\times n}$ have a uniform matroid.  The
  $K$-polynomial of $X_v$ is
  \[
  \K(X_v)=1 - \sum_{\substack{\lambda=(\lambda_1 \geq \lambda_2) \\
      2 \leq \lambda_2,\ \lambda_1 + \lambda_2 \leq n }}
  (-1)^{|\lambda|} s_\lambda(1,1) s_\lambda(u) e_{|\lambda|}(t).
  \]
\end{proposition}
\begin{proof}
  The degeneracy locus of the map $\psi_v$ defined by the matrix
  $x\odot v$ is $X_v$ by ---.  Therefore, the ideal $I_v$ of $X_v$ is
  resolved $G$-equivariantly by the Eagon-Northcott complex
  $\C_\bullet(\psi_v) \to I_v \to 0$, wherein
  \[
  \C_m(\psi_v) = \Sym^{m-4} \left(\End_R(R^2)\right) \otimes
  \bw^m R^n , \quad m = 4,5,\dots,n.
  \] 
  This is a minimal resolution since
  $\operatorname{depth}(I_v) = \operatorname{codim}(I_v)=n-3$
  \cite[Theorem A2.10]{eisenbud}, as we have shown above.

  When $R^n$ and $R^2$ are graded by characters of the diagonal torus
  in $G$ acting on~$R$, all the maps in $(\bw^2 R^2) \otimes_R
  \C_\bullet \to I_v \to 0$ are linear maps. To compute the
  $K$-polynomial of the terms of the resolution it suffices to compute
  the character of the $G$-module
  \[
  \left(\bw^2 \kk^2 \otimes \Sym^m \End(\kk^2)\right) \otimes
  \bw^m \kk^n.
  \]
  The character of the $\GL_2(\kk)$-module $\Sym^m \End(\kk^2)$ has been computed by
  D\'esarm\'enien, Kung and Rota \cite{rota} as
  \[
  \sum_{\lambda =(\lambda_1 \geq \lambda_2) \vdash m }
  s_\lambda(1,1)s_\lambda(u_1,u_2).
  \]
  The proposition follows.
\end{proof}

The case when $v \in \AA^{r \times n}$, $r>2$, has matroid $U_{2,n}$
is dealt with in the next section.

\section{Stabilization}\label{sec:stabilization}
The basic operation we consider in this section is embedding a
$G$-invariant subvariety $X \subset \AA^{ r \times n}$ in a matrix
space with more rows, and stabilizing it under the larger general
linear group action. We extend this notion to certain equivariant
coherent modules, and describe what it does at the level of
$K$-polynomials.

\subsection{The structure of $R$ as a $G$-module}
We need to recall the decomposition of the ring $R$ as a module for
the group $G$. This decomposition can be gleaned from the Cauchy
identity,
\[
\Hilb(R) = \prod_{\substack{1 \leq i \leq r\\1 \leq j \leq n}} \frac{1}{1-u_i t_j} = \sum_{\lambda} s_\lambda(u_1,\dots,u_r) s_\lambda(t_1,\dots,t_n),
\]
which decomposes $R$ as module for $\GL_r(\kk) \times \GL_n(\kk)$. It gives the
irreducible decomposition,
\[
R \approx \bigoplus_{\lambda} \SS^{\lambda}(\kk^r)^\vee \otimes \SS^\lambda(\kk^n),
\]
the direct sum running over partitions $\lambda$ with at most $\min
\{r,n\}$ parts. To obtain the irreducible decomposition as a
$G$-module one takes the weight space decomposition of
$\SS^\lambda(\kk^n)$ and obtains
\[
R \approx \bigoplus_{\lambda, \tau} \SS^{\lambda}(\kk^r)^\vee \otimes \kk_{\operatorname{cont}(\tau)},
\]
where the sum is over partitions $\lambda$ with at most $r$ parts and
semistandard Young tableaux $\tau:\lambda \to \{1,2,\dots,n\}$. Here
$\kk_{\operatorname{cont}(\tau)}$ is the one-dimensional
representation of $T$ with action $t \cdot 1 = \prod_{i=1}^n
t_i^{\operatorname{cont}(\tau)_i}$.

The ring $R$ has the $\kk$-linear basis of standard bitableaux, which
we now describe. If the shapes of $\sigma$ and $\tau$ are both
$(1^\ell)$, $\ell \leq \min\{r,n\}$, then the \textbf{bitableau}
$(\sigma,\tau)$ is the determinant of the square submatrix of $x =
(x_{ij})$ whose rows are selected by the entries of $\sigma$ and
whose columns are selected by the entries of $\tau$, both in order. When
$\sigma$ and $\tau$ have the same shape and more than one column, one
takes the product of the bitableaux obtained from corresponding pairs
of columns of $\sigma$ and $\tau$. A \textbf{standard bitableau} is a
bitableaux $(\sigma,\tau)$ where both $\sigma$ and $\tau$ are
semistandard Young tableaux.

A $G$ \textbf{lowest weight vector} of $R$ is a $\ZZ^r \times
\ZZ^n$-graded homogeneous polynomial that is fixed by the subgroup of
$\GL_r(\kk)$ consisting of upper-triangular matrices with $1$'s on the
diagonal.  Explicitly, the standard bitableaux $(\sigma,\tau)$ where
row $i$ of $\sigma$ is filled with the number $i$ form a basis for the
space of lowest weight vectors of~$R$. Every irreducible $G$-module in
$R$ possesses a unique lowest weight vector which is a linear
combination of these special bitableaux. It is important in what
follows that every lowest weight vector in $R$ whose weight is
$\lambda$ only involves those variables $x_{ij}$ with $i \leq
\ell(\lambda)$.

\subsection{Stabilization of modules} 
For the rest of this section we need to emphasize the number of rows
of the matrices we are working with. As such, we write $R_r$ for
$\kk[x_{ij}: i \in [r], j\in [n]]$ and $G_r$ for $\GL_r(\kk) \times T$. We
have a tower of $\kk$-algebras
\[
 R_1 \subset R_2 \subset  \cdots \subset R_r \subset R_{r+1} \subset \cdots
\]
given by adding a row of indeterminates. 

We also have a tower of groups
\[
G_1 \subset G_2 \subset \cdots \subset G_r \subset G_{r+1} \subset \cdots
\]
given by taking the direct sum with a $1$-by-$1$ identity matrix. The
two towers are compatible in the sense that the $G_r$ lowest weight
vectors of $R_r$ whose weight $\lambda$ has $s$ non-zero parts are
contained in $R_s$, and are $G_s$ lowest weight vectors therein. It
follows that if we take the smallest $G_{r+1}$ representation in
$R_{r+1}$ containing a fixed $G_r$ representation within $R_r$, the
characters are the same in the sense that one is obtained from the
other by replacing the Schur polynomial
$s_\lambda(u_1^{-1},\dots,u_r^{-1})$ with
$s_\lambda(u_1^{-1},\dots,u_{r+1}^{-1})$.

Let $V$ be a finite dimensional representation of the torus $T$. We
can view $V$ as a representation of $G_r$ by letting the $\GL_r(\kk)$
factor of $G_r$ act trivially. Consider the equivariant finite free
module $R_r \otimes V$ and a $G_r$-equivariant submodule $N_r$
thereof. There are inclusions
\[
N_r \subset R_r \otimes V \subset R_{r+1} \otimes V ,
\]
and we define $N_{r+1}$ to be the smallest $G_{r+1}$-equivariant
$R_{r+1}$-module satisfying $N_r \subset N_{r+1} \subset R_{r+1}
\otimes V$.
Let $J_{r}$ be the ideal in $R_{r}$ generated by the size
$r$ minors of the coordinate matrix $[x_{ij}]$.
\begin{proposition}\label{prop:rho character}
  Suppose that the character of $N_r$ is $$\sum_{\lambda \in \Par_r,\a \in \ZZ^n}
  d_{\lambda,\a} s_\lambda(u^{-1}_1,\dots,u^{-1}_r) t^\a.$$ Then, the
  character of $N_{r+1} + (J_{r+1} \otimes V)$ is
  \[
  \sum_{\lambda \in \Par_r,\a \in \ZZ^n}
  d_{\lambda,\a} s_\lambda(u^{-1}_1,\dots,u^{-1}_r,u^{-1}_{r+1}) t^\a.
  \]
  plus the character of $J_{r+1} \otimes V$.
\end{proposition}
\begin{proof}
  Since $V$ is a sum of trivial representations of $\GL_r(\kk) \subset
  G_r$, every lowest weight vector of $R_{r+1} \otimes V$ whose weight
  $\lambda$ satisfies $\lambda_{r+1} \neq 0$ is contained in $J_{r+1}
  \otimes V$. All of these lowest weight vectors are contained in
  $N_{r+1} + (J_{r+1} \otimes V)$. Any other lowest weight vector of
  $N_{r+1} + (J_{r+1} \otimes V)$ has a weight $\lambda$ which
  satisfies $\lambda_{r+1} = 0$. Such lowest weight vectors are
  contained in $R_r \otimes V$. Intersecting $N_{r+1} + (J_{r+1}
  \otimes V)$ with $R_r \otimes V$ returns $N_r$, and thus every such
 lowest weight vector must be a lowest weight vector of $N_r$.
\end{proof}

Let $M$ be a $G_r$-equivariant $R_r$-module with an equivariant
presentation of the form
\begin{equation}\label{stabilization}
  0 \to N_r \to R_r \otimes V \to M \to 0.  
\end{equation}
Define a \textbf{stabilization} of $M$ to be $G_{r+1}$-equivariant
$R_{r+1}$-module $\stab(M)$ making the following sequence exact:
\[
0 \to N_{r+1} + (J_{r+1} \otimes V) \to R_{r+1} \otimes V \to \stab(M)
\to 0.
\]

\begin{example}
  A stabilization of $R_r$ is $R_{r+1}/J_{r+1}$. More generally, a
  stabilization of a quotient $R_r/I_r$ is $R_{r+1}/(I_{r+1} +
  J_{r+1})$, where $I_{r+1}$ is the smallest $G_{r+1}$ invariant ideal
  in $R_{r+1}$ containing $I_r$.
\end{example}

Given a $G_r$-stable closed subvariety $X \subset \AA^{r \times n}$,
we let $\stab(X)$ denote the smallest $G_{r+1}$-stable closed
subvariety of $\AA^{(r+1) \times n}$ that contains $X$, namely
$\cl{G_{r+1} X}$. 
\begin{lemma}\label{lem:priming sheaf priming support}
  Let $X$ be a $G_r$-stable closed subvariety of $\AA^{r \times
    n}$. Then, the coordinate ring of $\rho(X)$ is a stabilization of
  the coordinate ring of $X$.
\end{lemma}
\begin{proof}
  Denote the ideal defining $X$ by $I_r \subset R_r$. Suppose we have
  a lowest weight vector of $R_{r+1}$ that vanishes on $\rho(X)$. If
  it has weight $\lambda$ satisfying $\lambda_{r+1} = 0$ then it must
  lie in $R_r \subset R_{r+1}$. It follows that this lowest weight
  vector is in $I_r$, since $\rho(X)\cap\AA^{r\times n} = X$. 
  If the weight $\lambda$ has $\lambda_{r+1} \neq
  0$ then it is in $J_{r+1}$. Since $\rho(X)$ contains no rank $r+1$
  matrices, its ideal contains $J_{r+1}$ which contains all lowest
  weight vectors whose weight $\lambda$ satisfies $\lambda_{r+1} \neq
  0$. It follows that the ideal defining $\rho(X)$ is $I_{r+1} +
  J_{r+1}$.
\end{proof}

\subsection{Stabilization of $K$-polynomials}\label{ssec:K stabilization}
Although stabilization of modules is not unique, its effect on
$K$-classes is. This is most easily understood at the level of Hilbert
series.
\begin{lemma}\label{lem:row of zeros}
  Let $M$ be a $G_r$-equivariant $R_r$-module, with a presentation as
  in \eqref{stabilization}. Let $\rho(M)$ be any stabilization of
  $M$. Write
\begin{equation}\label{eq:row of zeros expansion}
  \Hilb( M ) =
  \sum_{\lambda \in \Par_r, \a \in \NN^n} d_{\lambda, \a} s_\lambda(u_1,\dots,u_r) t^\a.
\end{equation}
  Then,
  \[
  \Hilb(\rho(M))= \sum_{\lambda \in \Par_r, \a \in \NN^n}
  d_{\lambda,\a} s_\lambda(u_1,\dots,u_r,u_{r+1}) t^\a.
  \]
\end{lemma}
\begin{proof}
  The character of $R_{r+1}$ is the character of $J_{r+1}$ plus the
  character of $R_{r+1}/J_{r+1}$. The latter is obtained from that of
  $R_r$ by replacing each $s_\lambda(u_1^{-1},\dots,u_r^{-1})$ with
  $s_\lambda(u_1^{-1},\dots,u_{r+1}^{-1})$. Now appeal to the
  additivity of characters along exact sequences and
  Proposition~\ref{prop:rho character}.
\end{proof}
We proceed to define a collection of linear operators
\[
\rho_k : \ZZ[u_1,\ldots,u_r]^{\Ss_r}\to
\ZZ[u_1,\ldots,u_r,u_{r+1}]^{\Ss_{r+1}},
\] 
which we extend to maps $\ZZ[u,t^{\pm 1}]^{\Ss_r}\to
\ZZ[u,u_{r+1},t^{\pm 1}]^{\Ss_r}$
by letting them act
linearly on the $t$ variables.  
Using the extension of the notation $s_\lambda$ in Section~\ref{sec:from algebra},
$\rho_k$ can be concisely defined by
\[
\rho_k\,s_\lambda(u_1,\ldots,u_r) := s_{\lambda,k}(u_1,\ldots,u_r,u_{r+1})
\]
when $\lambda$ is a partition with $r$ parts, possibly including zero parts.
Recall that, using the determinantal formula, this means
\[
(\rho_ks_\lambda)(u_1,\ldots,u_{r+1}) =
\frac{\det(u_i^{\lambda_j+r+1-j})_{i,j=1,\ldots,r+1}}{\det(u_i^{r+1-j})_{i,j=1,\ldots,r+1}}\]   
where $\lambda_{r+1} = k$.  Alternatively,
$\rho_ks_\lambda(u)$ equals $(-1)^\ell s_\mu(u,u_{r+1})$ if $\mu$ is a
partition containing $\lambda$ such that the skew shape
$\mu\setminus\lambda$ is a ribbon
with $\ell+1$ non-empty rows whose leftmost box is in row
$r+1$
;
and $\rho_ks_\lambda(u)=0$ if there is no such~$\mu$.
\begin{example}
  We see that $\rho_k 1 = 0$ for $0 < k < r$ and $\rho_r 1 = (-1)^r
  s_{(1^{r+1})}(u_1,\dots,u_{r+1})$. In general $\rho_0
  s_\lambda(u_1,\dots,u_r) = s_\lambda(u_1,\dots,u_r,u_{r+1})$.
\end{example}

We also collect these operators $\rho_k$ into a sum
\[\rho = \sum_{k=0}^n (-1)^k e_k(t)\, \rho_k: \ZZ[u,t^{\pm 1}]^{\Ss_r}\to \ZZ[u,u_{r+1},t^{\pm1}]^{\Ss_{r+1}}.\]

We will sometimes abuse notation and allow $\rho$
to denote the analogous operator on the ring of symmetric polynomials
in $r-1$ variables. The argument of $\rho$ makes clear
which operator is being referred to.

\begin{proposition}\label{prop:row of zeros}
  If $M$ has a presentation as in \eqref{stabilization}, then,
\begin{equation}\label{eq:row of zeros 1}
  \K(\rho( M)) = \rho\,\K(M).
\end{equation}
\end{proposition}
In particular, if $X$ is a closed subvariety of $\AA^{r\times n}$,
then $\K(\rho(X)) = \rho\,\K(X)$.

\begin{proof}
  Consider the following sum, corresponding to a single Schur polynomial
  in $\K(M)$
  within the right side of~\eqref{eq:row of zeros 1}:
\[\sum_{k=0}^n e_k(-t) \rho_ks_\lambda(u_1,\ldots,u_r).\]
We expand along the last row the numerator in our determinantal
definition of $\rho_k$, corresponding to the introduced
$\lambda_{r+1}=k$.  This turns the displayed sum above into
\[
\sum_{k=0}^ne_k(-t)\frac{\displaystyle \sum_{\ell=1}^{r+1}
  (-1)^{\ell-1} u_\ell^{k}
  \det(u_i^{\lambda_j+r+1-j})_{i=1,\ldots,\widehat\ell,\ldots,r+1}^{j=1,\ldots,r}}
{\det(u_i^{r+1-j})_{i,j=1}^{r+1}}.\] Moving the inner summation
outside, we may rewrite the sum over $k$ as a product, yielding
\begin{align*}
&\mathrel{\phantom{=}}
\sum_{\ell=1}^{r+1}
\prod_{j=1}^n (1-u_\ell t_j) \frac{\displaystyle (-1)^{\ell-1}
\det(u_i^{\lambda_j+r+1-j})_{i=1,\ldots,\widehat\ell,\ldots,r+1}^{j=1,\ldots,r}}
{\det(u_i^{r+1-j})_{i,j=1}^{r+1}}
\\&= 
\sum_{\ell=1}^{r+1}
\prod_{j=1}^n (1-u_\ell t_j)
\frac{(\prod_{i\neq\ell} u_i)\, (-1)^{\ell-1}
\det(u_i^{\lambda_j+r-j})_{i=1,\ldots,\widehat\ell,\ldots,r+1}^{j=1,\ldots,r}}
{\prod_{k<i} (u_k-u_i)}
\\&= 
\sum_{\ell=1}^{r+1}
\prod_{j=1}^n (1-u_\ell t_j)
\frac{(\prod_{i\neq\ell} u_i)\,
s_\lambda(u_1,\ldots,\widehat{u_\ell},\ldots,u_{r+1})}
{\prod_{i\neq\ell}(u_\ell-u_i)}
\end{align*}
using the Vandermonde identity.
Since $\lambda$ only appears in this expression in a single $s_\lambda$,
in which the rest of the expression is linear, and since the $\rho$ are
linear in the $t$ variables as well, the whole right side of~\eqref{eq:row of zeros 1}
equals
\begin{align*}
&\mathrel{\phantom{=}}
\sum_{\ell=1}^{r+1}
\prod_{j=1}^n (1-u_\ell t_j)
\frac{(\prod_{i\neq\ell} u_i)\,
\Hilb(M;u_1,\ldots,\widehat{u_\ell},\ldots,u_{r+1}, t)
\prod_{i\neq\ell}\prod_j (1-u_i t_j)}
{\prod_{i\neq\ell}(u_\ell-u_i)}
\\&= \left(\prod_{i=1}^{r+1}\prod_{j=1}^n (1-u_it_j)\right)
\sum_{\ell=1}^{r+1}
\frac{(\prod_{i\neq\ell} u_i)\,
\Hilb(M;u_1,\ldots,\widehat{u_\ell},\ldots,u_{r+1}, t)}
{\prod_{i\neq\ell}(u_\ell-u_i)},
\end{align*}
where we have also used the definition of the $K$-polynomial.

By the same determinantal manipulations used above,
expanding along the last row, this is equal to
\[
\left(\prod_{i=1}^{r+1}\prod_{j=1}^n (1-u_it_j)\right)
\sum_{\lambda,\mathbf a} d_{\lambda,\mathbf a}s_\lambda(u_1,\ldots,u_r,u_{r+1})t^{\mathbf a}
\]
if $\Hilb(M)$ is given the expansion in coefficients
from~\eqref{eq:row of zeros expansion}.
By Lemma~\ref{lem:row of zeros}, this is
\[\left(\prod_{i=1}^{r+1}\prod_{j=1}^n (1-u_it_j)\right)
\Hilb( \rho(M); u_1,\dots,u_{r+1},t),\]
which equals $\K(\rho(M))$.  This proves \eqref{eq:row of zeros 1}.
\end{proof}

\section{Direct sum, parallel extension and
  $K$-polynomials}\label{sec:matroid constructions}
In this section we consider some common matroid operations as
operations on vector configurations, and see how these manifest
themselves at the level of $K$-classes. 
\subsection{Direct sum}\label{sec:direct sum} In order to understand
how direct sums interact with $K$-classes we first consider a
concatenation operation.
\begin{proposition}\label{prop:direct sum} Suppose that $v^1 \in
\AA^{r \times n_1}$ has its last $r'$ rows equal to zero and $v^2 \in
\AA^{r \times n_2}$ has its first $r-r'$ rows equal to zero. Let $v =
(v^1,v^2)$ be the concatenation of $v^1$ and $v^2$. Then, the
$K$-polynomial of $v$ is the product of the $K$-polynomials of $v^1$
and $v^2$.
\end{proposition}
Combining this result with Proposition~\ref{prop:row of zeros} allows
us to see how the direct sum of vector configurations manifests itself
at the level of $K$-polynomials.
\begin{corollary}\label{cor:direct sum}
  Let $v^1$ and $v^2$ be vector configurations in $\AA^{r_1 \times
    n_1}$ and $\AA^{r_2 \times n_2}$, and $v^1 \oplus v^2$ their
  direct sum in $\AA^{r \times n}$. Then,
  \[
  \K(X_{v^1 \oplus v^2}) = \rho^{r_2}
    \K(X_{v^1})\cdot \rho^{r_1}
    \K(X_{v^2})
  \]
\end{corollary}

\begin{proof}[Proof of Proposition~\ref{prop:direct sum}] 
There is a projection of $\AA^{r \times n}$ onto the
first $n_1$ columns and the last $n_2$ columns. The orbit of $v$ is
the intersection of the pullbacks of the orbits of $v^1$ and $v^2$
under these projections. It follows that the ideal of the orbit of $v^1 \oplus v^2$
is the sum of the inclusions of the ideals of $v^1$ and $v^2$ into $R$.
Denote these ideals and their inclusion by $I_{v^1}$ and
$I_{v^2}$.  Since these ideals are in different variables
we conclude that
  \[ \Hilb(R/I_{v^1 \oplus v^2}) = \Hilb(R/(I_{v^1} + I_{v^2})) = \Hilb(R/I_{v^1}
\otimes R/I_{v^2})
  \] and hence that the $K$-polynomial of $R/I_v$ is
  \[ \Hilb(R/I_{v^1}) \Hilb(R/I_{v^2}) \prod_{i=1}^r\prod_{j=1}^n
(1-u_i t_j)
  \] The $K$-polynomial of $R/I_{v^1}$ is
  \[\K(R_1/I_{v^1})\prod_{i=1}^r \prod_{j={n_1+1}}^n (1-u_i t_j),\]
where $R_1$ is the coordinate ring of $\AA^{r \times n_1}$. Similarly,
the $K$-polynomial of $R/I_{v^2}$ is
  \[\K(R_2/I_{v^2};u,t) \prod_{i=1}^r \prod_{j=1}^{n_1} (1-u_i t_j),\]
and from this the result follows.
\end{proof}

\subsection{Parallel extension}\label{ssec:parallel extension}
Here we are concerned with the effect on the $K$-polynomial of 
duplicating a column of $v\in\AA^{r\times(n-1)}$, which corresponds to
a parallel extension of the underlying matroid.  
In view of Section~\ref{sec:direct sum}, it is just as 
informative to compare the matrix with duplicated column 
to a matrix of the same size with one of the duplicated columns 
replaced by zero.  This gives the next theorem a
particularly nice form.

Let $\delta_{n-1}$ be the $(n-1)$th Demazure operator on the $t$ variables,
given by
\[\delta_{n-1}(f) = \frac{t_{n-1}f - t_n\sigma_{n-1}f}{t_{n-1}-t_n}\]
for $f\in\ZZ[t_1,\ldots,t_n,u_1,\ldots,u_r]$, where
$\sigma_{n-1}\in\Ss_n$ is the transposition $(n\!-\!1\ n)$, \linebreak
and $\Ss_n$ acts by permuting the $t$ variables.

\begin{theorem}\label{thm:parallel}
  Suppose that the last two columns of $v^\|\in\AA^{r\times n}$ 
  are nonzero and equal, and $v\in\AA^{r\times n}$ is obtained from $v^\|$
  by changing the last column to zero.  Then 
  \[\K(X_{ v^\|}) = \delta_{n-1}\K(X_v)\]
\end{theorem}

This theorem comes quickly from a $\PP^1$-bundle construction like the
one used for Schubert varieties \cite{bgg,demazure}.  To extend the
parallelism of these two situations, Schubert varieties $\Omega_\lambda$ in
$G(r,n)$ are in bijection with Schubert matroids $M_\lambda$, in
such a way that a generic point of $\Omega_\lambda$ has matroid
$M_\lambda$.  If the indexing partition $\lambda$ of one Schubert
matroid satisfies $\lambda_1=n-r-1$, and
$\lambda'=(n-r,\lambda_2,\ldots,\lambda_r)$ is obtained from it by
adding a box, then $n$ is a loop in $M_\lambda$ and $M_{\lambda'}$ is
a parallel extension of $M_\lambda\setminus\{n\}$, while
$\K(X_{\lambda'})=\delta_{n-1}\K(X_\lambda)\in K^0(G(r,n))$.

The precise statement we will use is the following ``sweeping lemma''.
It is a $K$-theoretic analogue of the cohomological lemma \cite[Lemma~2.2.1]{knutson notes},
whose statement we have mimicked closely.
\begin{lemma}\label{lem:sweeping}
  Let $P\subseteq B\subseteq T$ be a triple of Lie groups such that $P/B\cong\PP^1$ and $T$
  is a torus acting with weight $\mu$ on $\mathfrak p/\mathfrak b$.  
  Let $r\in N_p(T)$ be an element of the normalizer, inducing an automorphism
  $r$ of $T^\ast$ such that $r\cdot\mu=-\mu$.
  
  Let $V$ be a $P$-representation and $X\subseteq V$ a $B$-invariant subvariety. Then
  in $K_0^T(V)$ we have the equality
  \[d[P\cdot X]=\frac{[X]-\chi (r\cdot[X])}{1-\chi}\]
  where $d$ is the degree of the map $P\mathbin{\times^B}X\to P\cdot X$ 
  (or $0$ if $Y$ is $P$-invariant).
\end{lemma}
The proof goes through as in~\cite{knutson notes}, except that for a torus $T$
acting on $\PP^1$ via the weight $\mu$, the relation that attains
in $K_0^T(\PP^1)$ is
\[ [\{0\}] - \mu[\{\infty\}] = 1-\mu, \]
as can still be checked by equivariant localization.

\begin{proof}[Proof of Theorem~\ref{thm:parallel}]
  In any matrix $w\in X_{v^\|}$,
  the last two columns $w_{n-1}$ and $w_n$ are parallel. 
  Let $a\in\kk^n$ be one of $w_{n-1}$ and $w_n$ which is nonzero,
  if either is, or 0 if $w_{n-1}=w_n=0$.  Then 
  $(w_1,\ldots,w_{n-2},a,0)$ is in $X_v$, 
  and we can write $w=(w_1,\ldots,w_{n-2},y_1a,y_2a)$
  for some generically unique choice of $(y_1:y_2)\in\PP^1(\kk)$.  

  Let the variety $X\subseteq\AA^{r\times n}$ be $X_v$. 
  Since $v_n=0$, this orbit is in fact $B$-equivariant where
  \[B = \GL_r\times(\kk^\times)^{n-2}\times
    \left\{\begin{bmatrix}*&0\\{}*&*\end{bmatrix}\right\}  
    \supseteq \GL_r\times T=G\] 
  and the $\{\scriptsize\begin{bmatrix}*&0\\{}*&*\end{bmatrix}\}$
  factor acts on the last two columns.  
  Let $P=\GL_r\times (\kk^\times)^{n-2}\times\GL_2$, 
  so that as above $P\cdot X = X_{v^\|}$, and the map
  $P\mathbin{\times^B}X\to P\cdot X$ is degree~1.
  Take the $T$ of the lemma to be the maximal torus in $G$.
  Then $\mu=t_n/t_{n-1}$, and we will take $r$ to be 
  $(1,1,\scriptsize\begin{bmatrix}0&1\\1&0\end{bmatrix})$,
  whose action is the same as that of~$\sigma_{n-1}$.
  Then the theorem is immediate from Lemma~\ref{lem:sweeping}.
\end{proof}

We can now combine Proposition~\ref{prop:K-polynomial of U_2,n} and
Theorem~\ref{thm:parallel} to obtain an explicit formula for the
$K$-polynomial of an arbitrary $v \in \AA^{2 \times n}$. We will only
formulate our result for those $v$ of full rank with no columns equal
to zero.  If the $j$\/th column of $t$ is zero, 
then $\K(X_v)$ is simply the $K$-class
where this column is deleted multiplied by $(1-u_1t_j)(1-u_2t_j)$.

If $v \in \AA^{2 \times n}$ has no zero columns then we
define its \textbf{parallelism partition} to be the decreasing
sequence of sizes of its rank one flats.
\begin{proposition}\label{prop:K rank 2}
  Suppose that $v \in \AA^{2 \times n}$ has rank two and no zero
  columns. Write
  \[
  \K(X_v;u,t) = 
  \sum_{\substack{\b \in \{0,1\}^n\\0 \leq k \leq |\b|/2}} d_{k,\b}(v) s_{(|\b|-k,k)}(u) t^\b,
  \]
  as we may. Then,
  \begin{enumerate}
  \item $d_{0,\mathbf{0}}(v) = 1$, and $d_{0,\mathbf{b}}(v) = 0$ for all
    other $\b \in\{0,1\}^n$.
  \item If $k=1$ and $v_\b$ has rank one then
    $d_{k,\b}(v)=(-1)^{|\b|+1}$.
  \item If $k=1$ and the rank of $v_\b$ is two then $d_{k,\b}(v)=0$.
  \item If $k \geq 2$ and $v_\b$ has parallelism partition
    $\mu=(\mu_1 \geq \dots \geq \mu_\ell)$, $\ell \geq 4$ and
    $\mu_1'+\dots+\mu_{k-1}'\geq 2k-1$ then $d_{k,\b}(v) =
    (-1)^{|\b|+1}(\mu_1'+\dots + \mu_{k-1}' - 2k +1)$. Otherwise,
    $d_{k,\b}(v)=0$.
  \end{enumerate}
\end{proposition}
\begin{proof}
  We express Theorem~\ref{thm:parallel} in the following way: If $v'
  \in \AA^{2 \times (n+1)}$ is obtained by duplicating the last column
  of $v \in \AA^{2 \times n}$ then
  \[
  K(\cl{\GL_2 v' T^{n+1}}) = \delta_n \left( K(\cl{\GL_2 v T^n})(1-u_1t_{n+1})(1-u_2 t_{n+1}) \right).
  \]
  Temporarily write $K(v)$ and $K(v')$ for the $K$-polynomials of the
  orbit closures of $v$ and $v'$.

  Since, by induction, $K(v)$ is square free in the $t$-variables, we
  may uniquely write $K(v)= K(v)_0 + K(v)_1 t_n$, where $t_n$ does not
  appear in $K(v)_0$.  A simple computation yields
  \begin{align*}
    \delta_n (1-u_1 t_{n+1})(1-u_2 t_{n+1}) 
    &= 1- u_1u_2 t_n t_{n+1},\\
    \delta_n t_n(1-u_1t_{n+1})(1-u_2 t_{n+1})
    &= t_n + t_{n+1} - (u_1+u_2)t_n t_{n+1},
  \end{align*}
  and it follows that
  \begin{multline}\label{rank2recur}
      K(v') = K (v)_0 - s_{(1,1)}(u) t_n t_{n+1}K(v)_0\\
  + K(v)_1 t_n + K(v)_1 t_{n+1} - s_{(1)}(u) t_n t_{n+1}K(v)_1.
  \end{multline}
  We conclude that $K(v')$ is square free in the $t$-variables, does
  not contain any Schur polynomials of partitions of length $1$, and
  the coefficient of any $t^\b$ in $K(v')$ that does not contain {\it
    both} $t_n$ and $t_{n+1}$ is as described in the proposition.

  Suppose that $t^\b = t^\a t_n t_{n+1}$, $k \geq 1$, and write
  $(\a,1)$ for the exponent vector of $t^\a t_n$. Then
  \eqref{rank2recur} implies that
  \[
  d_{k,\b}(v')
  = -\left( d_{k-1,\a}(v) 
    + d_{k-1,(\a,1)}(v) + d_{k,(\a,1)}(v)
  \right).
  \]
  Since $d_{k,(\a,1)}(v) = d_{k,(\a,1)}(v')$, by our computation
  above, this yields
  \begin{align}\label{rank2recur2}
  d_{k,\b}(v') + d_{k,(\a,1)}(v') = - \left( d_{k-1,\a}(v) 
    + d_{k-1,(\a,1)}(v) \right).
  \end{align}
  What follows from here is a tedious check that the coefficients
  described in the proposition obey \eqref{rank2recur2}. To preserve
  the reader's patience, we will not provide the details of all possible
  cases, which are many. This will be forgiven since the case $k=1$ is
  addressed in much greater generality by Theorem~\ref{thm:hooks in
    K}.
  
  We will focus on the least degenerate case, when $k \geq 3$ and
  $d_{k-1,\a}(v)$, $d_{k-1,(\a,1)}(v)$ and $d_{k,(\a,1)}(v) =
  d_{k,(\a,1)}(v')$ are all non-zero. In this case, induction yields
  \begin{align*}
    d_{k-1,\a}(v) &= (-1)^{|\a|+1} \left(\mu_1'(\a) + \dots +\mu_{k-2}'(\a) - 2k+3\right),\\
    d_{k-1,(\a,1)}(v) &= (-1)^{|\a|+2}\left( \mu_1'(\a,1) + \dots + \mu_{k-2}'(\a,1) -2k +3 \right),\\
    d_{k,(\a,1)}(v') &= (-1)^{|\a|+2}\left( \mu_1'(\a,1) + \dots + \mu_{k-2}'(\a,1)+\mu_{k-1}'(\a,1) -2k +1 \right).
  \end{align*}
  Here, $\mu(\a)$ and $\mu(\a,1)$ are the parallelism partitions of
  $v_\a$ and $v_{(\a,1)}$. Now, $\mu(\a)$ and $\mu(\a,1)$ differ in
  exactly one position. Hence, the sum of the first two terms is
  $(-1)^{|\a|+2}$, unless the number of vectors parallel to $v_n$ in
  $v_\a$ is $k-1$ or larger. If the latter happens then the
  $d_{k-1,\a}(v) + d_{k-1,(\a,1)}(v)=0$. This implies that
  $d_{k,\b}(v')$  differs from $d_{k,(\a,1)}(v')$ by $\pm 1$ or $0$,
  according to whether the number of vectors parallel to $v_n$ in
  $v_{(\a,1)}$ is larger than $k-1$, or not. Hence $d_{k,\b}(v')$ is
  given by the given by the formula
  \[
  \mu_1'(\b) + \dots + \mu_{k-1}'(\b) - 2k +1,
  \]
  where $\mu(\b)$ is the parallelism partition of $v_\b$. The
  remaining cases are left to the reader.
\end{proof}

\section{The tensor module}\label{sec:tensor module}
The \textbf{tensor module} of $v \in \AA^{r \times n}$ is the cyclic
$\GL_r(\kk)$-module in $(\kk^r)^{\otimes n}$ generated by
\[
v_1 \otimes \dots \otimes v_n.
\]
We denote the tensor module of $v$ by $G(v)$. In this section we
consider the connection between the tensor module, the $K$-class, and
the matroid of $v$.

\subsection{$\OO(1,\dots,1)$ and the tensor module}\label{ssec:OO(1,...,1)}
Let $(\AA^{r \times n})^{\rm nz}$ denote the space of matrices in
$\AA^{r \times n}$ with no column equal to zero. This is a principal
$T$-bundle over the product of projective spaces $(\PP^{r-1})^n$. We
let $j: (\AA^{r \times n})^{\rm nz} \to \AA^{r \times n}$ denote the
inclusion, and $p:(\AA^{r \times n})^{\rm nz} \to (\PP^{r-1})^n$ the
projection. The tensor module $G(v)$ can be constructed from the line
bundle $\OO(1,\dots,1)$ on $(\PP^{r-1})^n$, which is the external
tensor product of the $\OO(1)$'s on each factor.

The inverse image $j^{-1} X_v$ is the intersection of
$X_v$ with $(\AA^{r \times n})^{{\rm nz}}$ and the
projection of this to $(\PP^{r-1})^n$ is the $\GL_r(\kk)$-orbit closure of
$p(v)$.
\begin{proposition}\label{prop:character of tensor module}
  For $v \in (\AA^{r \times n})^{{\rm nz}}$, $G(v)$ is dual as a
  $\GL_r(\kk)$-module to the global sections of $\OO(1,\dots,1)|_{\cl{\GL_r(\kk)
      p(v)}}$.  The character of $G(v)$, as a $\GL_r(\kk)$-module, is the
  coefficient of $t_1\cdots t_n$ in $\Hilb(X_v)$.
\end{proposition}
\begin{proof}
  The dual of $G(v)$ consists of those multilinear polynomials defined
  on $\GL_r(\kk) p(v)$. This proves the first claim.  The second follows
  since the character of $\OO(1,\dots,1)|_{\cl{\GL_r(\kk) p(v)}}$ is
  obtained from the coefficient in $\ZZ[u_1 ,\ldots,u_r ]$ of $t_1
  \cdots t_n$ in the multigraded Hilbert series of $X_v$ by
  replacing $u$ with $1/u$. Taking the dual representation at the
  level of characters corresponds to replacing each $u_i$ with
  $1/u_i$. The second claim follows.
\end{proof}

We will occasionally need to use facts about the tensor module of $v$
and all of its parallel extensions. As such, we set up the notation
for this now.  Given $\b \in \NN^n$, we let $v_\b$ denote the vector
configuration obtained from $v$ by duplicating the $i$th column $\b_i$
times (or omitting it if $\b_i = 0$). Let $v_\b^\otimes$ be tensor
product of the vectors in the configuration $v_\b$, in order, and
define $G(v_\b)$ to be the cyclic $\GL_r(\kk)$-module in $(\kk^r)^{\otimes
  |\b|}$ generated by $v_\b^\otimes$. The obvious generalization of
Proposition~\ref{prop:character of tensor module} is true for $v_\b$.
\begin{proposition}\label{prop:character of tensor module2}
  The character of $G(v_\b)$ is the coefficient of $t^\b$ in $\Hilb(X_v)$.
\end{proposition}

\subsection{Support of the tensor module}\label{ssec:support}
The \textbf{support} of the tensor module is the collection of
partitions of $n$ that index the irreducible representations appearing
in the irreducible decomposition of $G(v)$.  The \textbf{rank partition} of a
matroid $M$ is the sequence $\lambda(M(v)) =(\lambda_1,\lambda_2,\lambda_3,\dots)$
determined by the condition that
\[
\lambda_1 + \lambda_2 + \dots + \lambda_k
\]
is the size of the largest union of $k$ independent sets in $M$.
\begin{theorem}[Dias da Silva \cite{dds}]
  The rank partition of $M$ is a partition. If $M$ has no loops then there is
  a set partition of the ground set of $M$ into independent sets of
  sizes $\mu = (\mu_1 \geq \mu_2 \geq \dots \geq \mu_\ell)
  \vdash n$ if and only if $\mu \leq \lambda(M)$ in dominance order.
\end{theorem}
The following result is related to a generalization of Gamas's theorem
on the vanishing of symmetrized tensors (see \cite{berget,dds}).
\begin{proposition}\label{prop:support}
  The tensor module $G(v)$ has an irreducible submodule of lowest
  weight $\mu$ if and only if $\mu \geq
  \lambda(M(v))^t$. Further, $\mu \geq \lambda(M(v))^t$ if and only if
  there is a standard Young tableau of shape $\mu$ whose columns
  index independent sets of $M(v)$.
\end{proposition}
As an immediate corollary we have the following result.
\begin{corollary}\label{cor:support}
  Given $\b \in \NN^n$, the coefficient of $s_\mu(u ) t^\b$ in the
  multigraded Hilbert series of $X_v$ is positive if and only if
  $\mu \geq \lambda(M(v_\b))^t$. The latter condition happens if and
  only if there is a semi-standard Young tableaux of shape $\mu$
  and content $\b$ whose columns index independent sets of $M(v)$.
\end{corollary}

\subsection{Schur-Weyl duality}\label{ssec:Schur-Weyl duality}
One can study the irreducible decomposition of the tensor module using
the representation theory of the symmetric group.

We will denote the cyclic $\Ss_n$-module in $(\kk^r)^{\otimes n}$
generated by $v_1 \otimes \dots \otimes v_n$ by $\Ss(v)$.
\begin{proposition}\label{prop:schurWeyl}
  The tensor module $G(v)$ is Schur--Weyl dual to $\Ss(v)$. That is,
  there are isomorphisms,
  \[
  \Hom_{\GL_r(\kk)}(G(v),(\kk^r)^{\otimes n}) \cong \Ss(v) \qquad
  \Hom_{\Ss_n}(\Ss(v),(\kk^r)^{\otimes n}) \cong G(v),
  \]
  of $\Ss_n$-modules and $\GL_r(\kk)$-modules, respectively.
\end{proposition}
\begin{proof}
  Either of the above isomorphisms is defined as
  \[
  \varphi \mapsto \varphi(v_1 \otimes \dots \otimes v_n).
  \]
  Such a homomorphism, say $\varphi : G(v) \to (\kk^r)^{\otimes n}$,
  extends to a map of $\GL_r(\kk)$-modules $\tilde\varphi
  \in \End_{\GL_r(\kk)}((\kk^r)^{\otimes n})$. Schur--Weyl duality \cite[Section~6.2]{fultonHarris} asserts
  that $\tilde\varphi \in \kk\Ss_n$. It follows that
  \[
  \varphi(v_1 \otimes \dots \otimes v_n) \in \Ss(v).
  \]
  The other isomorphism is proved similarly.
\end{proof}
Combining Propositions~\ref{prop:character of tensor module2} and
\ref{prop:schurWeyl} we obtain a second proof of Lemma~\ref{lem:row of zeros}. Indeed
the isomorphism type of $\Ss(v_\b)$ visibly does not change when we
embed $v$ in a matrix space with more rows.

The irreducible representations of $\Ss_n$ that can appear in
$(\kk^r)^{\otimes n}$, and hence $\Ss(v)$, are indexed by partitions
of $n$ with at most $r$ parts. The irreducible representations of
$\GL_r(\kk)$ that can appear in $G(v)$ are indexed by the exact same set of
partitions, as we have discussed.
\begin{corollary}\label{cor:schurWeyl}
  For any partition $\lambda$ of $n$, the multiplicity of $\lambda$ in
  $\Ss(v)$ is equal to the multiplicity of $\lambda$ in $G(v)$.
\end{corollary}
\begin{proof}
  By another formulation of Schur--Weyl duality
  \cite[Section~6.1]{fultonHarris}, the functors
  $\Hom_{\Ss_n}(-,(\kk^r)^{\otimes n})$ and
  $\Hom_{\GL_r(\kk)}(-,(\kk^r)^{\otimes n})$ take an irreducible indexed by
  $\lambda$ to an irreducible indexed by $\lambda$. Since these
  functors commute with direct sums we are done.
\end{proof}

As a first application of Proposition~\ref{prop:character of tensor
  module} and Corollary \ref{cor:schurWeyl}, we extract from
Proposition~\ref{prop:K-polynomial of U_2,n} the character of the
tensor module for the uniform matroid in rank~2.

\begin{corollary}\label{cor:tensor module of U_2,n}
  Let $v \in \AA^{2 \times n}$ have uniform matroid. 
  The character of the tensor module $G(v)$ is
  \[
  s_{(n,0)}(u) + \sum_{\ell=1}^{n/2} (n-2\ell + 1)
  s_{(n-\ell,\ell)}(u).
  \]
\end{corollary}
\begin{proof}
  To see this we take the coefficient of $e_n(t) = t_1\cdots t_n$ in
  the product of the $K$-polynomial $\K(X_v;u,t)$ from
  Proposition~\ref{prop:K-polynomial of U_2,n} with $1/\prod_{i=1}^n
  (1-u_1 t_i)(1-u_2 t_i)$. Writing
  \[
  \K(X_v;u,t) = 1+ p_4(u) e_4(t) - p_5(u) e_5(t) - \dots +(-1)^{n}
  p_n(u) e_n(t),
  \]
  we see that the coefficient of $e_n(t)$ in the product in question is
  \[
  (u_1+u_2)^n + \binom{n}{4}(u_1 + u_2)^{n-4} p_4(u) - \binom{n}{5}(u_1 + u_2)^{n-5}
  p_5(u) +\dots + (-1)^n p_n(u).
  \]
  Setting $u_1 = u_2 = 1$ in this formula tells us the dimension of
  $G(v)$. We use the fact that $p_i(1,1) = (-1)^{i+1}\sum_{k=2}^{i/2}
  (i-2k+1)^2 = (-1)^{i+1} \binom{i-1}{3}$. From this we obtain
  \[
  \dim G(v) = 2^n + \sum_{i=4}^n (-1)^{i+1} \binom{n}{i}\binom{i-1}{3} 2^{n-i}
  = (n^3 + 5n + 6)/6.
  \]
  Since the multiplicity of the Specht module indexed by $(n-k,k)$ in
  $(\kk^2)^{\otimes n}$ is at most $(n-2k+1)$, by Schur--Weyl duality,
  we see that the multiplicity of $(n-k,k)$ in $G(v)$ is likewise
  bounded. Also, the multiplicity of $(n)$ in $G(v)$ is at most one
  since $\Sym^n(\kk^2)$ is an irreducible $\GL_2(\kk)$ module. If any of
  these multiplicities were less than these trivial upper bounds we
  would have
  \[
  (n^3 + 5n + 6)/6 < (n+1)+\sum_{i=1}^{n/2} (n-2k+1)^2 = (n^3 + 5n + 6)/6.
  \]
  It follows that each multiplicity is as large as possible in $G(v)$,
  thus proving the proposition.
\end{proof}

Using Proposition~\ref{prop:K rank 2}, it is possible to extract the
isomorphism type of the tensor module of any rank two configuration $v
\in \AA^{2 \times n}$ with no zero columns. In practice, the
computation becomes an endless checking of cases. We state the result
here, referring the reader to \cite[Theorem~3.5.1]{ABthesis} for a
proof avoiding the technology of $K$-polynomials, and relying further
on Schur--Weyl duality.
\begin{proposition}\label{prop:tensor module rank 2}
  Let $v \in \AA^{2 \times n}$ have rank two and no columns equal to
  zero. Let $\mu =(\mu_1 \geq \mu_2 \geq \dots)$ denote the
  parallelism partition of $v$. Then, the character of the tensor
  module of $v$ is
  \[
  s_{(n,0)}(u) + \sum_{k=1}^{n/2} \max(\mu_1'+\dots+ \mu_k'-2k+1,0 ) s_{(n-k,k)}(u).
  \]
\end{proposition}

\section{Hook shapes}\label{ssec:hooks}
A partition is called a \textbf{hook} if it has at most one part that
is not equal to one. The multiplicity of a hook shape $\lambda$ in
$\Ss(v)$ (equivalently, $G(v)$) is determined by the subcomplex of
non-broken circuit sets of $M(v)$. This is used to determine the certain coefficients of the $K$-polynomial $\K(X_v)$.

\subsection{The tensor module}
For a matroid $M$ with ground set contained in a totally ordered set,
a \textbf{broken circuit} of $M$ is a containment minimal dependent
set with its smallest ordered element removed.  A set is said to be an
\textbf{nbc set} if does not contain any of the broken circuits of
$M$. The collection of nbc sets of $M$ forms a subcomplex of $M$ whose
structure is well studied. The nbc sets of $M(v)$ are known to be
intimately related to the cohomology ring of the complement in
$(\kk^r)^*$ of the hyperplanes given by the vanishing of the linear
functionals $v_1,\dots v_n$ on $(\kk^r)^*$. In particular, enumerating the
nbc sets by corank yields the coefficients of the Poincar\'e
polynomial of this variety.

Here is our main result relating hook shapes and non-broken circuits.
\begin{theorem}\label{thm:hook in tensor module}
  The multiplicity of $(n-k+1,1^{k-1})$ in $\Ss(v)$ is the number of
  nbc bases of the truncation of $M(v)$ to rank $k$, if $k \leq
  \rk(M(v))$. It is zero otherwise.
\end{theorem}
We will let $\lambda_{n,k}$ denote the hook shape that is a partition
of $n$ with length $k$, i.e., $\lambda_{n,k} = (n-k+1,1^{k-1})$.
\begin{proof}
The element
\[
\sum_{\substack{\sigma \in \Ss_{[k]}\\\tau \in \Ss_{[n]\setminus [k]}}}
(-1)^{\ell(\sigma)} \sigma \tau \in \kk\Ss_n
\] 
acts as a projector from $\kk\Ss_n$ to the sum of two irreducible
Specht modules, one of shape $\lambda_{n,k+1}$, 
the other of shape $\lambda_{n,k}$. This follows from the Pieri rule
and the fact that the above element is a product of a row symmetrizer
and a column anti-symmetrizer.

Since $\Ss(v)$ is a cyclic module, it follows that the sum of the
multiplicities of $\lambda_{n,k+1}$ and $\lambda_{n,k}$ in it is the
dimension of the vector space
\[
\Ss(v)\sum_{\substack{\sigma \in \Ss_{[k]}\\\tau \in \Ss_{[n]\setminus [k]}}}
(-1)^{\ell(\sigma)} \sigma \tau \subset \bw^k (\kk^r) \otimes
\Sym^{n-k}(\kk^r) \subset (\kk^r)^{\otimes n}.
\]
The image of this space is spanned by the tensors
\[
\left\{\bigwedge_{i \in I} v_i \otimes \prod_{j \notin I} v_j : I \in
  \binom{[n]}{k}\right\}
\]
If $k=r$ then the wedges simply record whether $I$ is a basis of
$M(v)$. In this case we can forget the wedges and simply look at the
dimension of the vector space spanned by
\[
\left\{\prod_{j \notin B} v_j : B \in \B(M(v))\right\} \subset \Sym^{n-r}(\kk^r),
\]
where $\B(M(v))$ denotes the bases of $M(v)$. By a result of Orlik and
Terao (reproved as \cite[Corollary~2.3]{tutte}) the dimension of this
vector space is the number of nbc bases of $M(v)$, which agrees with
the statement of the theorem, since the hook $\lambda_{n,r+1}$ does
not appear in $\Ss(v)$ (even in $(\kk^r)^{\otimes n}$).

In case $k<r$ we project $v$ onto a generic $k$-dimensional subspace
through the origin, to obtain a new configuration $v'$. The
multiplicity of $\lambda_{n,k}$ in $\Ss(v')$ is the number of nbc
bases of the truncation of $M$ to rank $k$. Since $\Ss(v')$ is a
homomorphic image of $\Ss(v)$ this gives a lower bound for the
multiplicity of the length $k$ hook in $\Ss(v)$. It follows from
\cite[Theorem~5.4]{whitney} that this multiplicity in $\Ss(v)$ is at
most the number of nbc bases of the truncation of $M(v)$ to rank $k$
and from this the theorem follows.
\end{proof}

The Tutte polynomial of a matroid $M$ is the unique polynomial
$T_M = T_M(x,y) \in \ZZ[x,y]$ satisfying the conditions:
\begin{itemize}
\item[T1.] $T_M(x,y) = x$ if $M$ is rank zero on one
  element and $T_M(x,y) = y$ if $T$ is rank one on one element.
\item[T2.] $T_{M \oplus N} = T_M T_N$.
\item[T3.] If $e$ is neither a loop nor an isthmus of
  $M$, then $T_M = T_{M\setminus e} + T_{M/e}$.
\end{itemize}
It is well known that the Tutte evaluation $q^{\rk(M)}T_M(1+1/q,0)$ is
the generating function for the nbc sets of $M$ by their rank, as is
seen by appealing to the deletion contraction recurrence (T3).
\begin{corollary}\label{cor:hook hilbert}
  Let $d_{\lambda_{n,k},\mathbf{1}}$ denote the multiplicity of
  $\lambda_{n,k}$ in $G(v)$. Then
  \[
  \sum_{k = 0}^r d_{\lambda_{n,k},\mathbf{1}}q^{k-1}(q+1)= q^{\rk(M(v))}
  T_{M(v)}(1+1/q,0).
  \]
\end{corollary}
\begin{proof}
  First, $q^{\rk(M(v))}T_{M(v)}(1+1/q,0)$ is the generating function
  for nbc sets of $M(v)$ by their rank. Next, the number of nbc bases
  of the truncation of $M(v)$ to rank $k$ plus the number of nbc bases
  of the truncation of $M(v)$ to rank $k+1$ is the number of nbc sets
  of $M(v)$ of size $k$. To see this add the element $1$ to each size
  $k$ nbc set of $M(v)$ that does not already contain it. The sets
  that already contained $1$ were the nbc bases of the truncation to
  rank $k$, the other sets correspond to the nbc bases of the
  truncation to rank $k+1$.
\end{proof}



Reformulating the above result in terms of the multigraded Hilbert series of $X_v$ yields the following result.
\begin{corollary}\label{cor:tutte2}
Write
\[
\Hilb(X_v) = \sum_{\lambda \in \Par_r, \b \in \NN^n}
  d_{\lambda,\b} s_\lambda(u) t^\b.
\]
Then,
\[
\sum_{k=0}^r d_{\lambda_{n,k},\b} q^{k-1}(q+1) = q^{\rk(M(v_\b))}
T_{M(v_\b)}(1+1/q,0).
\]  
\end{corollary}
\begin{proof}

  The coefficient in question is the multiplicity of
  $\lambda_{|\b|,k}$ in $G(v_\b)$. Since this is the multiplicity of
  $\lambda_{|\b|,k}$ in $\Ss(v)$, by Corollary~\ref{cor:schurWeyl},
  the result follows from Proposition \ref{prop:character of tensor module2} and Corollary~\ref{cor:hook hilbert}.
\end{proof}

\subsection{Hooks and the multivariate Tutte polynomial}\label{ssec:multivariate hooks}
In this section we use the multivariate Tutte polynomial of Sokal to
describe the hook shapes that appear in the $K$-polynomial of a matrix
orbit closure.

We start with a definition, given a matroid $M$ with ground set $[n]$
we define
\[
\widetilde{Z}_M(q;t_1,\dots,t_n) = \sum_{\b \in \{0,1\}^n} q^{-\rk(M|\b)} t^b.
\]
This is the \textbf{multivariate Tutte polynomial} of $M$, due to
Sokal \cite{sokal}, which is also known to statistical physicists as
the $q$-state Potts model partition function.

\begin{lemma}[Ardila--Postnikov {\cite[Lemma~6.6]{ardilaPostnikov}}]\label{lem:ardilaPostnikov}
  The generating function for the Tutte polynomials of $M|\b$, as
  $\b$ ranges over $\NN^n$, can be expressed as
  \begin{multline*}
    \sum_{\b \in \NN^n} (x-1)^{-\rk(M|\b)} T_{M|\b}(x,y)t^\b\\ =
    \frac{1}{\prod_{j=1}^n
      (1-t_j)}\widetilde{Z}_M\left( (x-1)(y-1);\frac{(y-1)t_1}{1-yt_1},\dots,\frac{(y-1)t_n}{1-yt_n} \right).
  \end{multline*}
\end{lemma}

In what follows we will work in the subring $S$ of
$\ZZ[u_1,\dots,u_r,t_1,\dots,t_n]^{\Ss_r}$ where the $u$-degree of a
polynomial equals its $t$-degree.

Using Lemma~\ref{lem:row of zeros} we can unambiguously extend the
Hilbert series of $X_v$ to a \textit{symmetric function} in the
infinitely many variables $u_1,u_2,\dots$, with coefficients in
$\ZZ[[t_1,\dots,t_n]]$. From this symmetric function we mod out those
Schur functions in $u$ that are not hook shapes. It is a consequence
of the Littlewood--Richardson rule that this quotient of $S$ is
isomorphic, as a ring, to $\ZZ[[q, t_1 ,\dots,t_n]]$. The image of
$s_{\lambda_{|\b|,k}}(u)t^\b$ under this isomorphism is
$q^{k-1}(q+1)t^\b$.

We take the image of the Hilbert series $\Hilb(X_v)$ in
$\ZZ[[q,t_1,\dots,t_n]]$, which has the form
\[
\sum_{\b \in \NN^n}\sum_{k=0}^\infty d_{\lambda_{|b|,k}}
q^{k-1}(q+1)t^\b =
\sum_{\b \in \NN^n}q^{\rk(M(v)|\b)}T_{M(v)|\b}(1+1/q,0) t^\b,
\]
the second equality being Corollary~\ref{cor:tutte2}.
Lemma~\ref{lem:ardilaPostnikov} 
then condenses the sum to
\[
\frac{1}{\prod_{j=1}^n(1-t_j)} \widetilde{Z}_{M(v)}( -q^{-1}  ;
-t_1,\dots,-t_n).
\]

We can now state our result on hook shapes in terms of the
$K$-polynomial of $X_v$.
\begin{proposition}\label{prop:fullEnumeratorHooks}
  The enumerator of hook shapes in the $K$-polynomial of $X_v$ is
  \[
  \prod_{j=1}^n \frac{1 -(q+1)t_j + q(q+1)t_j^2 - q^2(q+1)t_j^3 +\cdots }{(1-t_j)} \widetilde{Z}_{M(v)}( -q^{-1} ; -t_1,\dots,-t_n).
  \]
  in the following sense: The coefficient of $q^{k-1}(q+1)t^\b$, $k
  \leq r$, is equal to the coefficient of
  $s_{\lambda_{|\b|,k}}(u_1,\dots,u_r) t^\b$ in the $K$-polynomial.
\end{proposition}
\begin{proof}
  This follows from the definition of $K$-polynomial, after we push
  the denominator of the Hilbert series
  $\prod_{j=1}^n\prod_{i=1}^\infty(1-u_i t_j)$ into
  $\ZZ[[q,t_1,\dots,t_n]]$.
\end{proof}
\begin{theorem}\label{thm:hooks in K}
  Take $\b \in \{0,1\}^n$. The coefficient of
  $s_{\lambda_{|\b|,k}}(u)t^\b$ in the $K$-polynomial of $X_v$ is $(-1)^k$ if $\b$ indexes a rank $k-1$ dependent set of
  $M(v)$. It is zero otherwise.
\end{theorem}
\begin{proof}
  Since we are only interested in the square-free monomials in the
  $t$'s, we work modulo $\<t_1^2,\dots,t_n^2\>$.  By
  Proposition~\ref{prop:fullEnumeratorHooks}, the image of the
  $K$-polynomial in $\ZZ[[q,t_1,\dots,t_n]]/\<t_1^2,\dots,t_n^2\>$ is
  \[
  \prod_{j=1}^n \frac{1-t_j(1+q)}{(1-t_j)} \widetilde{Z}_{M(v)}( -q^{-1} ;
  -t_1,\dots,-t_n) \equiv \prod_{j=1}^n (1-t_j q) \widetilde{Z}_{M(v)}( -q^{-1} ;
  -t_1,\dots,-t_n).
  \]
  Give the right-hand side of this equality the name
  $\FakeDep_{M(v)}(q;t_1,\dots,t_n)$. By \cite[Eq.~(4.18a)]{sokal} we
  see that $\FakeDep$ satisfies the recurrence
  \[
  \FakeDep_{M(v)} = (1-t_i q)\FakeDep_{M(v-v_i)}
  + t_i q \FakeDep_{M(v/v_i)}, \quad (v_i \neq 0).
  \]
  We define the multivariate polynomial
  \[
  \Dep_{M(v)}(q;t_1,\dots,t_n) = 1+\sum_{\substack{\b \in \{0,1\}^n,
      \b \neq (0,\dots,0)\\\rk M(v_\b) < |\b| }} (-1)^{\rk M(v_\b)}
  q^{\rk M(v_\b)-1}(q+1) t^\b.
  \]
  The sum is over the dependent sets of $M(v)$. It is straight forward
  that $\Dep$ satisfies the same recurrence as $\FakeDep$. Further, if
  $v_i$ is the zero vector then
  \[
  \Dep_{M(v)}(q;t_1,\dots,t_n) = (1-qt_i)  \Dep_{M(v-v_i)}(q;t_1,\dots,t_n),
  \]
  and likewise for $\FakeDep$. Since both $\FakeDep$ and $\Dep$
  evaluate to $1$ when $v=(v_1)$, $v_1 \neq 0$, and to $(1-t_1q)$ when
  $v = (0)$ it follows that they are equal in general. The theorem
  follows by taking the coefficient of $q^{\rk(M(v_\b))-1}(q+1) t^\b$ in
  $\Dep_{M(v)}$.
\end{proof}

\subsection*{Acknowledgements}
We thank
Dave Anderson, 
Calin Chindris, 
Ezra Miller, 
Leonardo Mihalcea, 
Richard Rim\'anyi, 
Seth Sullivant 
and Alex Woo 
for useful discussions,
as well as the anonymous referees who pointed out errors in 
earlier claimed proofs of Conjecture~\ref{conj:main}.

\def\cprime{$'$}

\end{document}